\documentclass[11pt, twoside, leqno]{article}

\usepackage{amssymb}
\usepackage{amsmath}
\usepackage{amsthm}
\usepackage{mathrsfs}

\allowdisplaybreaks

\def\rr{{\mathbb R}}
\def\rn{{{\rr}^n}}
\def\zz{{\mathbb Z}}
\def\cc{{\mathbb C}}
\def\nn{{\mathbb N}}

\def\ca{{\mathcal A}}

\def\ccc{{\mathcal C}}
\def\cd{{\mathcal D}}

\def\cm{{\mathcal M}}

\def\mr{{\mathcal R}}
\def\cs{{\mathcal S}}

\def\cx{{\mathcal X}}

\def\mi{{\mathrm I}}
\def\mj{{\mathrm J}}
\def\mh{{\mathrm H}}

\def\fz{\infty}
\def\az{\alpha}
\def\bz{\beta}
\def\dz{\delta}
\def\bdz{\Delta}
\def\ez{\epsilon}
\def\gz{{\gamma}}
\def\bgz{{\Gamma}}

\def\lz{\lambda}
\def\blz{\Lambda}

\def\boz{{\Omega}}
\def\tz{\theta}
\def\sz{\sigma}

\def\lf{\left}
\def\r{\right}

\def\la{\langle}
\def\ra{\rangle}
\def\hs{\hspace{0.25cm}}
\def\ls{\lesssim}
\def\gs{\gtrsim}

\def\pa{\partial}
\def\ov{\overline}
\def\noz{\nonumber}
\def\wz{\widetilde}
\def\wh{\widehat}
\def\ev{\equiv}
\def\st{\subset}
\def\com{\complement}
\def\bh{\backslash}
\def\gfz{\genfrac{}{}{0pt}{}}
\def\iint{\int\hspace{-0.2cm}\int}

\def\dist{\mathop\mathrm{\,dist\,}}
\def\supp{\mathop\mathrm{\,supp\,}}

\def\cro{\ccc_\rho}
\def\mz{{\frac n2(\frac 1{p_\Phi^-}-\frac12)}}
\def\mzx{{\frac n2(\frac 1{\wz p_\Phi}-\frac12)}}
\def\hl{{\mathrm H}_l}

\def\rb{{\rho(\mu(B))[\mu(B)]^{1/2}}}
\def\rkb{{\rho(\mu(2^kB))[\mu(2^kB)]^{1/2}}}

\def\dxt{\,\frac{d\mu(x)\,dt}{t}}
\def\dyt{\,\frac{d\mu(y)\,dt}{t}}
\def\dt{\,\frac{dt}{t}}
\def\ds{\,\frac{ds}{s}}
\def\cub{\chi_{U_k(B)}}
\def\lp{{L^p(\cx)}}
\def\lx{{L^\Phi(\cx)}}
\def\xt{{\cx\times(0,\fz)}}
\def\ml{{\cm^{M,\ez}_\Phi(L)}}
\def\mlx{{\cm^{M,\ez}_\Phi(L^*)}}
\def\mly{{\cm^{M_1}_{\Phi,L}(\cx)}}
\def\tx{{T_\Phi(\cx)}}
\def\txx{{(T_\Phi(\cx))^*}}
\def\txz{{T_\Phi^\fz(\cx)}}
\def\txv{{T_{\Phi,{\rm v}}^\fz(\cx)}}
\def\txl{{T_{\Phi,0}^\fz(\cx)}}
\def\txy{{T_{\Phi,1}^\fz(\cx)}}
\def\txb{{T_{2,b}^2(\cx)}}
\def\txe{{T_{2}^2(\cx)}}
\def\ttx{{\wz T_\Phi(\cx)}}
\def\dfb{\dz_k(f,B)}
\def\tml{(t^2L)^Me^{-t^2L}}
\def\tmlx{(t^2L^*)^Me^{-t^2L^*}}
\def\tmy{(t^2L)^{M_1}e^{-t^2L}}

\def\Iem{\lf(I-e^{-r_B^2L}\r)^M}
\def\Ilm{\lf(I-[I+r_B^2L]^{-1}\r)^M}
\def\iemx{(I-e^{-r_B^2L^{\ast}})^M}
\def\Iemx{\lf(I-e^{-r_B^2L^{\ast}}\r)^M}
\def\tl{t^2Le^{-t^2L}}
\def\tlx{t^2L^*e^{-t^2L^*}}
\def\vy{{V_{k,1}(B)}}
\def\ve{{V_{k,2}(B)}}
\def\pme{(\Phi,\,M,\,\ez)_L}
\def\pml{(\Phi,\,\wz M,\,\ez)_{L^*}}
\def\pmx{(\Phi,\,M,\,\ez)_{L^*}}
\def\bmol{{\mathop\mathrm{\,BMO}^M_{\rho,L}(\cx)}}

\def\vmol{{\mathop\mathrm{\,VMO}^M_{\rho,L}(\cx)}}
\def\tvmol{{\mathop\mathrm{\wz{\,VMO}}^M_{\rho,L}(\cx)}}
\def\vmo{{\mathop\mathrm{\,VMO}_{\rho,L}(\cx)}}
\def\bmo{{\mathop\mathrm{\,BMO}_{\rho,L}(\cx)}}
\def\bmox{{\mathop\mathrm{\,BMO}_{\rho,L^*}(\cx)}}
\def\bmoxx{{(\bmox)^*}}
\def\vmox{{(\vmo)^*}}
\def\hx{{H_{\Phi,L}(\cx)}}
\def\hxx{{H_{\Phi,L^*}(\cx)}}
\def\hmfl{{H_{\Phi,\mathrm {fin},L}^{\mathrm {mol},\ez, M}(\cx)}}

\def\hmx{{H_{\Phi,L}^{M,\ez}(\cx)}}
\def\bmx{{B_{\Phi,L}^{M,\ez}(\cx)}}
\def\byx{{B_{\Phi,L}^{M_1,\ez_1}(\cx)}}
\def\bx{{B_{\Phi,L}(\cx)}}
\def\bxx{{B_{\Phi,L^*}(\cx)}}
\def\plm{\pi_{L,M}}
\def\ply{\pi_{L,1}}

\newtheorem{theorem}{Theorem}[section]
\newtheorem{lemma}{Lemma}[section]

\newtheorem{proposition}{Proposition}[section]
\newtheorem{remark}{Remark}[section]
\newtheorem{definition}{Definition}[section]

\numberwithin{equation}{section}

\begin{document}
\arraycolsep=1pt

\title{{\vspace{-5cm}\small\hfill\bf Kyoto J. Math. (to appear)}\\
\vspace{4.5cm}\Large Vanishing Mean Oscillation Spaces Associated with
Operators Satisfying Davies-Gaffney Estimates
\footnotetext{\hspace{-0.35cm} 2010 {\it
Mathematics Subject Classification}. Primary 42B35; Secondary 42B30, 46E30, 30L99.
\endgraf {\it Key words and phrases}. metric measure space,
operator, bounded $H_\infty$ functional
calculus, Davies-Gaffney estimate, Orlicz function, Orlicz-Hardy space,
BMO, VMO, molecule, dual.
\endgraf The second author is supported by the National
Natural Science Foundation (Grant No. 10871025) of China
and Program for Changjiang Scholars and Innovative
Research Team in University of China.}}
\author{Yiyu Liang, Dachun Yang and Wen Yuan\,\footnote{Corresponding author}}
\date{ }
\maketitle

\vspace{-0.5cm}

\begin{center}
\begin{minipage}{13cm}
{\small {\bf Abstract}\quad Let $(\mathcal{X}, d, \mu)$ be a metric measure
space, $L$ a linear operator which has a bounded $H_\infty$ functional
calculus and satisfies the Davies-Gaffney estimate, $\Phi$ a concave
function on $(0,\infty)$ of critical lower type $p_\Phi^-\in(0,1]$
and $\rho(t)\equiv t^{-1}/\Phi^{-1}(t^{-1})$ for all $t\in(0,\infty)$.
In this paper, the authors introduce the generalized VMO space
${\mathrm {VMO}}_{\rho,L}({\mathcal X})$ associated with $L$, and establish
its characterization via the tent space. As applications, the authors show that
$({\mathrm {VMO}}_{\rho,L}({\mathcal X}))^*=B_{\Phi,L^*}({\mathcal X})$,
where $L^*$ denotes the adjoint operator of $L$ in $L^2({\mathcal X})$ and
$B_{\Phi,L^*}({\mathcal X})$ the Banach completion of the Orlicz-Hardy space
$H_{\Phi,L^*}({\mathcal X})$.
}
\end{minipage}
\end{center}

\vspace{0.2cm}

\section{Introduction}\label{s1}

\hskip\parindent John and Nirenberg \cite{jn61} introduced the {\it
space ${\mathrm {BMO}\,}(\rn)$}, which is defined to be the space of
all $f\in L_{\rm loc}^1(\rn)$ such that
$$\|f\|_{\rm BMO\,(\rn)}\ev\sup_{{\rm ball}\ B\st\rn}\frac1{|B|}
\int_B|f(x)-f_B|\,dx<\fz,$$
where and in what follows, $f_B\ev\frac1{|B|}\int_Bf(x)\,dx$. The
space ${\mathrm {BMO}\,}(\rn)$ was proved to be the dual of the Hardy space
$H^1(\rn)$ by Fefferman and Stein in \cite{fs72}.

Sarason \cite{s75} introduced the {\it space ${\mathrm {VMO}\,}(\rn)$}, which
is defined to be the space of all $f\in{\mathrm {BMO}\,}(\rn)$ such that
$$\lim_{c\to0}\sup_{\gfz{\mathrm{ball}\,B\st\rn}{r_B\le c}}\frac1{|B|}
\int_B|f(x)-f_B|\,dx=0,$$
where $r_B$ denotes the radius of the ball $B$.
In order to represent $H^1(\rn)$ as a dual
space, Coifman and Weiss \cite{cw77} introduced the {\it space
${\mathrm {CMO}\,}(\rn)$}, which is defined to be the closure of all
infinitely differentiable functions with compact support in the
${\mathrm {BMO}\,}(\rn)$ norm and was originally denoted by the \emph{symbol}
${\mathrm {VMO}\,}(\rn)$ in \cite{cw77}, and proved that $({\mathrm
{CMO}\,}(\rn))^*=H^1(\rn)$. For more properties of ${\mathrm
{BMO}\,}(\rn)$, ${\mathrm {VMO}\,}(\rn)$ and ${\mathrm
{CMO}\,}(\rn)$, we refer the reader to Janson \cite{j80} and
Bourdaud \cite{b02}.

Let $L$ be a \emph{linear operator} in $L^2(\rn)$ that generates an
analytic semigroup $\{e^{-tL}\}_{t\ge0}$ with kernels satisfying an
\emph{upper bound of Poisson type}. The Hardy space $H_L^1(\rn)$,
the BMO space ${\mathrm {BMO}}_L(\rn)$ and Morrey spaces
associated with $L$ were
introduced and studied in \cite{adm,dy05b,dxy07}. Duong and Yan
\cite{dy05a} further proved that $(H_L^1(\rn))^*={\mathrm
{BMO}}_{L^*}(\rn)$, where and in what follows, $L^*$ denotes the
{\it adjoint operator} of $L$ in $L^2(\rn)$. Moreover, recently,
Deng et al. \cite{ddsty} introduced the \emph{space} ${\mathrm
{VMO}}_L(\rn)$, the space of vanishing mean oscillation associated
with operator $L$, and proved that $({\mathrm
{VMO}}_L(\rn))^*=H_{L^*}^1(\rn)$ and also
$${\mathrm {VMO}}_\Delta(\rn) ={\mathrm {CMO}}(\rn)={\mathrm
{VMO}}_{\sqrt\bdz}(\rn)$$
with equivalent norms. Let $\Phi$ on
$(0,\fz)$ be a continuous, strictly increasing, subadditive function
of upper type $1$ and of critical lower type $p_\Phi^-\le 1$
but near to $1$ (see Section \ref{s2.4} below for the definition).
Let $\rho(t)\equiv t^{-1}/\Phi^{-1}(t^{-1})$ for all $t\in(0,\fz)$.
A typical example of such Orlicz functions is $\Phi(t)\equiv t^p$
for all $t\in(0,\fz)$ and $p\le 1$ but near to $1$. Jiang and Yang
\cite{jya} introduced the VMO-type space ${\mathrm
{VMO}}_{\rho,L}(\rn)$ and proved that the dual space of ${\mathrm
{VMO}}_{\rho, L^*}(\rn)$ is the space $B_{\Phi, L}(\rn)$, where
$B_{\Phi,L}(\rn)$ denotes the \emph{Banach completion} of the Orlicz-Hardy space
$H_{\Phi,L}(\rn)$ in \cite{jyz09}.

Let $L$ be a \emph{second order divergence form elliptic operator}
with complex bounded measurable coefficients and $\Phi$ a \emph{continuous,
strictly increasing, concave function} of
critical lower type $p_\Phi^-\in (0,1]$.
Jiang and Yang \cite{jy10a} studied the VMO-type spaces
${\mathrm {VMO}}_{\rho,L}(\rn)$ and proved that
the dual space of ${\mathrm {VMO}}_{\rho, L^*}(\rn)$ is
the space $B_{\Phi, L}(\rn)$, where $B_{\Phi, L}(\rn)$
denotes the \emph{Banach completion} of the Orlicz-Hardy space
$H_{\Phi,L}(\rn)$ in \cite{jy10}. (We remark that
the \emph{assumptions on $p_\Phi$ in \cite{jy10, jy10a} can
be relaxed into the same assumptions on
$p_\Phi^-$}; see Remark \ref{r2.2}(ii) below.)
In particular, when $\Phi(t)
\equiv t$ for all $t\in (0,\fz)$, then $\rho(t)\ev1$ and
$({\mathrm {VMO}}_{1,L}(\rn))^*=H_{L^*}^1(\rn)$, which was also
independently obtained by Song and Xu \cite{sx},
where $H_{L^*}^1(\rn)$ denotes the Hardy space first
introduced by Hofmann and Mayboroda \cite{hm09} (see also
\cite{hm09c}).

Let $(\cx,d)$ be a \emph{metric space endowed with a doubling measure} $\mu$ and
$L$ a \emph{non-negative self-adjoint operator satisfying Davies-Gaffney
estimates}. Hofmann et al. \cite{hlmmy} introduced the Hardy space
$H_L^1(\cx)$ associated to $L$. Jiang and Yang \cite{jy} further
introduced the Orlicz-Hardy space $\hx$. Anh \cite{a10} studied the
VMO space ${\mathrm {VMO}}_L(\cx)$ associated to $L$ and proved that
the dual space of ${\mathrm {VMO}}_L(\cx)$ is the Hardy space
$H_L^1(\cx)$. Recently, Duong and Li \cite{dl} observed  that the
assumption ``$L$ is a non-negative self-adjoint operator'' in
\cite{hlmmy} can be replaced by a weaker assumption that ``\emph{$L$ has a
bounded $H_\fz$ functional calculus on $L^2(\cx)$}'' and introduced
the Hardy space $H_L^p(\cx)$ with $p\in (0,1]$, which was further
generalized by Anh and Li \cite{al11} to the Orlicz-Hardy spaces
$\hx$.

From now on, we always assume that {\it $L$ is a linear operator
which has a bounded $H_\fz$ functional calculus and satisfies
Davies-Gaffney estimates} and that {\it $\Phi$ is a continuous,
strictly increasing, concave function of critical lower
type $p_\Phi^-\in (0,1]$}. In this paper, we introduce the generalized
VMO space $\vmo$ associated with $L$, and establish its characterization
via the tent space in \cite{jy}. Then, we further prove that $\vmox=\bxx$,
where $\bxx$ denotes the \emph{Banach completion} of the Orlicz-Hardy space
$\hxx$ in \cite{al11}. When $\Phi(t)\equiv t$ for all $t\in(0,\fz)$,
we denote $\vmo$ simply by $\mathrm{\,VMO}_{L}(\cx)$. As a special
case of the main results in this paper, we show that
$({\mathrm{\,VMO}_{L}(\cx)})^*=H_{L^*}^1(\cx)$, which,
when $L$ is nonnegative self-adjoint, was already obtained by Anh \cite{a10}.

Precisely, the paper is organized as follows. In Section \ref{s2},
we recall some known notions and notation concerning metric measure
spaces $\cx$, then describe some basic assumptions on the considered
operator $L$ and the Orlicz function $\Phi$ and present some
properties of the operator $L$ and the Orlicz function $\Phi$ considered in
this paper.

In Section \ref{s3}, we first obtain the $\rho$-Carleson measure
characterization (see Theorem \ref{t3.1} below) of the space $\bmo$
in \cite{al11} via first establishing a Calder\'on reproducing formula
(see Proposition \ref{p3.3} below). Differently from
the Calder\'on reproducing formula in \cite[Proposition\,4.6]{jy},
the Calder\'on reproducing formula in
Proposition \ref{p3.3} below holds for
all molecules instead of atoms in \cite{jy}, which
brings us some extra difficulty due to the lack of the
support of molecules.
Then we introduce the generalized VMO space $\vmo$
associated with $L$, and the tent space $\txv$ and establish some
basic properties of these spaces. In particular, we characterize the
space $\vmo$ via $\txv$; see Theorem \ref{t3.4} below. To this end,
we first need make clear the dual relation between $\hxx$ and $\bmo$
(see Theorem \ref{t3.2} below), which is deduced from a technical
result on the optimal representation of finite linear combinations
of molecules (see Theorem \ref{t3.1} below). We remark that variants
of Theorems \ref{t3.1} and \ref{t3.2} below have already been
given respectively in \cite[Theorems 3.15, 3.13 and 3.16]{al11}
without a detailed proof of \cite[Theorem 3.15]{al11}.
We give a detailed proof of Theorem \ref{t3.1} below
which induces \emph{more accurate indices} appearing in
Theorems \ref{t3.1} and \ref{t3.2} below, comparing with
\cite[Theorems 3.13 and 3.15]{al11} (see Remark
\ref{r3.2} below). Moreover, the proof of
Theorem \ref{t3.1} below simplifies the proof of
\cite[Theorem 5.4]{hlmmy} in a subtle way, the proof of
\cite[Theorem 5.4]{hlmmy} strongly depends
on the support of atoms; see Remark \ref{r3.1} below.

In Section \ref{s4}, we first obtain, in Theorem \ref{t4.1} below,
the dual space of the tent space $\txv$ in Definition \ref{d3.4}
below, from which, we further deduce that $\vmox=\bxx$
in Theorem \ref{t4.2} below, where $\bxx$
denotes the \emph{Banach completion} of $\hxx$. In particular, we obtain
$({\mathrm{VMO}_{L}(\cx)})^*=H_{L^*}^1(\cx)$.

Finally we make some conventions on notation. Throughout the whole
paper, we denote by $C$ a {\it positive constant} which is
independent of the main parameters, but it may vary from line to
line. The {\it constant with subscripts}, such as $C_1$, does not
change in different occurrences. We also use $C(\gz,\cdots)$ to
denote {\it a positive constant depending on the indicated
parameters $\gz,$ $\cdots$}. The {\it symbol} $A\ls B$ means that
$A\le CB$. If $A\ls B$ and $B\ls A$, then we write $A\sim B$. We
also set $\nn\equiv\{1,\, 2,\, \cdots\}$ and
$\zz_+\equiv\nn\cup\{0\}$. The {\it symbol $B(x,r)$} denotes the
ball $\{y\in\cx:\ d(x,y)<r\}$; moreover, let $CB(x,r)\equiv
B(x,Cr)$. For a measurable set $E$,
denote by $\chi_{E}$ the {\it characteristic function} of $E$
and by $E^\com$ the \emph{complement} of $E$ in $\cx$.

\section{Preliminaries\label{s2}}

\hskip\parindent In this section, we first recall some notions and
notation on metric measure spaces and then describe some basic
assumptions on the considered operator $L$ in this paper and its
functional calculus; finally, we also present some basic assumptions
and properties on Orlicz functions.

\subsection{Metric measure spaces\label{s2.1}}

\hskip\parindent Throughout the whole paper, let $\cx$ be a \emph{set}, $d$
a \emph{metric} on $\cx$ and $\mu$ a \emph{nonnegative Borel regular measure} on
$\cx$. Moreover, assume that there exists a constant $C_1\ge1$ such
that for all $x\in\cx$ and $r>0$,
\begin{equation}\label{2.1}
V(x,2r)\le C_1V(x,r)<\fz,
\end{equation}
where $B(x,r)\equiv\{y\in\cx:\ d(x,y)<r\}$ and
\begin{equation}\label{2.2}
V(x,r)\equiv \mu(B(x,r)).
\end{equation}

Observe that if $d$ is further assumed to be a quasi-metric, then
$(\cx,d,\mu)$ is called a \emph{space of homogeneous type} in the sense of
Coifman and Weiss \cite{cw71} (see also \cite{cw77}).

Notice that the doubling property \eqref{2.1} implies the following
\emph{strong homogeneity property}: there exist some positive constants $C$
and $n$, depending on $C_1$, such that
\begin{equation}\label{2.3}
V(x,\lz r)\le C\lz^nV(x,r)
\end{equation}
uniformly for all $\lz\ge1$, $x\in\cx$ and $r>0$. The parameter $n$
measures the \emph{dimension} of the space $\cx$ in some sense. Also, there
exist constants $C\in (0,\fz)$ and $N\in [0, n]$, depending on
$C_1$, such that
\begin{equation}\label{2.4}
V(x, r)\le C\lf(1+\frac{d(x,y)}r\r)^NV(y,r)
\end{equation}
uniformly for all $x,y\in\cx$ and $r>0$. Indeed, the property \eqref{2.4}
with $N=n$ is a simple corollary of the strong homogeneity property
\eqref{2.3}. In the cases of Euclidean spaces, Lie groups of polynomial
growth and, more generally, Ahlfors regular spaces, $N$ can be chosen
to be $0$.

In what follows, for any ball $B\subset\cx$, we set
\begin{equation}\label{2.5}
U_0(B)\equiv B\quad {\rm and}\quad U_j(B)\equiv 2^jB\backslash2^{j-1}B
\quad{\rm for\ }j\in\nn.
\end{equation}

The following covering lemma established in \cite[Lemma 2.1]{a10} plays
a key role in the sequel.

\begin{lemma}\label{l2.1}
For any $\ell>0$, there exists $N_\ell\in\nn$,
depending on $\ell$, such that for all balls $B(x_B,\ell r)$,
with $x_B\in\cx$ and $r>0$, there exists a family
$\{B(x_{B, i},r)\}_{i=1}^{N_\ell}$ of balls such that

{\rm i)} $B(x_B,\ell r)\st\cup_{i=1}^{N_\ell}B(x_{B, i},r)$;

{\rm ii)} $N_\ell\le C\ell^n$;

{\rm iii)} $\sum_{i=1}^{N_\ell}\chi_{B(x_{B, i},r)}\le C$.

\noindent Here $C$ is a positive constant independent of $x_B$, $r$ and $\ell$.
\end{lemma}

\subsection{Holomorphic functional calculi\label{s2.2}}

\hskip\parindent We now recall some basic notions of holomorphic functional
calculi introduced by McIntosh \cite{m86}.

Let $0<\nu<\gz<\pi$. Define the {\it closed sector $S_{\nu}$} in the
complex plane $\cc$ by setting $S_{\nu}\equiv\{z\in\cc:\
|\mathrm{arg}\,z|\le\nu\}\cup\{0\}$ and denote by $S^0_{\nu}$ its
{\it interior}. We employ the following \emph{subspaces},
$H_\fz(S_\nu^0)$ and $\Psi(S_\nu^0)$, of the {\it space
$H(S^0_{\nu})$} of all holomorphic functions on $S^0_{\nu}$:
$$H_{\fz}(S^0_{\nu})\equiv\lf\{b\in H(S^0_{\nu}):\ \|b\|_{L^\fz(S^0_{\nu})}\equiv
\sup_{z\in S^0_{\nu}}|b(z)|<\fz\r\}$$ and
\begin{eqnarray*}
&&\Psi(S^0_{\nu})\equiv\{\psi\in H(S^0_{\nu}):\ \text{there exist}\
s\in(0,\fz)\ \text{and}\ C\in(0,\fz)\ \text{such that}\noz\\
&&\hspace{5em}\text{for all}\ z\in S^0_{\nu},\,|\psi(z)|\le C|z|^s
(1+|z|^{2s})^{-1}\}.
\end{eqnarray*}

Given $\nu\in(0,\pi)$, a closed operator $L$ in $L^2 (\rn)$ is
called to be of {\it type $\nu$} if $\sz(L)\subset S_{\nu}$, where $\sz(L)$
denotes its {\it spectra}, and if for all $\gz>\nu$, there
exists a positive constant $C_{\gz}$ such that for all $\lz\not\in
S_{\gz}$, $\|(L-\lz I)^{-1}\|_{L^2 (\rn)\to L^2 (\rn)}\le
C_{\gz}|\lz|^{-1}$. Let $\mathscr{X}$ and $\mathscr{Y}$ be two
\emph{linear normed spaces} and $T$ be a \emph{continuous linear operator} from
$\mathscr{X}$ to $\mathscr{Y}$. Here and in what follows,
$\|T\|_{\mathscr{X}\to\mathscr{Y}}$ denotes the {\it operator norm
of $T$ from $\mathscr{X}$ to $\mathscr{Y}$}. Let $\tz\in(\nu,\gz)$
and $\Gamma$ be the {\it contour} $\{\xi=re^{\pm i\tz}:\ r\ge0\}$
parameterized clockwise around $S_{\nu}$. Then if $L$ is of type
$\nu$ and $\psi\in\Psi(S^0_{\nu})$, the {\it operator} $\psi(L)$ is
defined by
$$\psi(L)\equiv\frac{1}{2\pi i}\int_{\Gamma}(L-\lz I)^{-1}\psi(\lz)\,d\lz,$$
where the integral is absolutely convergent in $\mathfrak{L}(L^2
(\rn), L^2 (\rn))$ (the {\it class of all bounded linear operators
in $L^2 (\rn)$}). By the Cauchy theorem, we know that
$\psi(L)$ is independent of the choices of $\nu$ and $\gz$ such
that $\tz\in(\nu,\gz)$. Moreover, if $L$ is one-to-one and has dense
range, and $b\in H_{\fz}(S^0_{\gz})$, then $b(L)$ is defined by
setting $b(L)\equiv[\psi(L)]^{-1}(b\psi)(L)$, where $\psi(z)\equiv
z(1+z)^{-2}$ for all $z\in S^0_{\gz}$. It was proved by McIntosh
\cite{m86} that $b(L)$ is a well-defined linear operator in $L^2 (\rn)$.
Moreover, the operator $L$ is said to have a {\it bounded
$H_{\fz}$-calculus} in $L^2 (\rn)$, if for all $\gz\in(\nu,\pi)$,
there exists a positive constant $\wz{C}_{\gz}$
such that for all $b\in H_{\fz}(S^0_{\gz})$, $b(L)\in
\mathfrak{L}(L^2 (\rn), L^2 (\rn))$ and
\begin{equation}\label{2.6}
\|b\|_{L^2 (\rn)\to L^2 (\rn)}\le
\wz{C}_{\gz}\|b\|_{L^\fz(S^0_{\gz})}.
\end{equation}

\subsection{Assumptions on the operator $L$\label{s2.3}}

\hskip\parindent Throughout the whole paper,
we always suppose that the considered operators
$L$ satisfy the following \emph{assumptions}.

\medskip

\noindent{\bf Assumption $(L)_1$.} The operator $L$ has a bounded
$H_\fz$-calculus in $L^2(\cx)$.

\medskip

\noindent{\bf Assumption $(L)_2$.} The semigroup $\{e^{-tL}\}_{t>0}$
generated by $L$ is analytic on $L^2(\cx)$ and satisfies the \emph{Davies-Gaffney
estimate}, namely, there exist positive constants $C_2$ and $C_3$ such that
for all closed sets $E$ and $F$ in $\cx$, $t\in(0,\fz)$ and $f\in L^2(E)$,
\begin{equation}\label{2.7}
\|e^{-tL}f\|_{L^2(F)}\le
C_2\exp\lf\{-\frac{[\dist(E,F)]^2}{C_3t}\r\}\|f\|_{L^2(E)},
\end{equation}
where and in what follows, $\dist(E,F)\equiv\inf_{x\in E,\,y\in
F}d(x,y)$ and the \emph{space} $L^2(E)$ denotes the set of all $\mu$-measurable
functions on $E$ such that
$\|f\|_{L^2(E)}\equiv\{\int_E|f(x)|^2\,d\mu(x)\}^{1/2}<\fz$.

\medskip

\begin{remark}\rm\label{r2.1}
By the functional calculus of $L$ on $L^2(\cx)$,
it is easy to see that if an operator $L$ satisfies Assumptions
$(L)_1$ and $(L)_2$, the adjoint operator $L^*$ also satisfies Assumptions
$(L)_1$ and $(L)_2$ and, therefore, the following Lemmas \ref{l2.2}
and \ref{l2.3} also hold for $L^*$.
\end{remark}

By Assumptions $(L)_1$ and $(L)_2$, we have the following
technical result which was obtained by Anh and Li
in \cite[Proposition 2.2]{al11}.

\begin{lemma}\label{l2.2}
Let $L$ satisfy Assumptions $(L)_1$ and $(L)_2$. Then for any
fixed $k\in\zz_+$ (resp. $j,k\in\zz_+$ with $j\le k$), the family
$\{(t^2L)^ke^{-t^2L}\}_{t>0}$ (resp.
$\{(t^2L)^j(I+t^2L)^{-k}\}_{t>0}$) of operators also satisfies the
Davies-Gaffney estimate \eqref{2.7} with positive constants $C_2$,
$C_3$ depending only on $n$ and $k$ (resp. $n$, $j$ and $k$).
\end{lemma}

By \eqref{2.6}, we have the following useful lemma.

\begin{lemma}\label{l2.3}
Let $L$ satisfy Assumptions $(L)_1$ and $(L)_2$. Then for any
fixed $k\in\nn$, the operator given by setting, for all
$f\in{L^2(\cx)}$ and $x\in\cx$,
$$S_L^kf(x)\ev\lf(\iint_{\bgz(x)}\lf|(t^2L)^ke^{-t^2L}f(y)\r|^2\,
\frac{d\mu(y)}{V(x,t)}\,\frac{dt}t\r)^{1/2},$$
is bounded on ${L^2(\cx)}$.
\end{lemma}

\subsection{Orlicz functions\label{s2.4}}

\hskip\parindent Let $\Phi$ be a positive function on
$\rr_{+}\equiv(0,\fz)$. The function $\Phi$ is called of {\it
upper \emph{(resp.} {\it lower}\emph{)} type $p$} for some $p\in[0,\fz)$,
if there exists a positive constant $C$ such that for all
$t\in[1,\fz)$ (resp. $t\in(0,1]$) and $s\in(0,\fz)$,
\begin{equation}\label{2.8}
\Phi(st)\le Ct^p \Phi(s).
\end{equation}
Obviously, if $\Phi$ is of lower type $p$ for some $p\in(0,\fz)$,
then $\lim_{t\to0_{+}}\Phi(t)=0$. So for the sake of convenience,
if it is necessary, we may \emph{assume} that $\Phi(0)=0$. If $\Phi$ is
of both upper type $p_1$ and lower type $p_0$, then $\Phi$ is called
of \emph{type} $(p_0,\,p_1)$. Let
\begin{eqnarray}\label{2.9}
&&p_{\Phi}^{+}\equiv\inf\{p\in(0,\fz):\ \text{there exists a
positive constant}\,
\,C\,\,\\
&&\hspace{4.5 em}\text{such that}\ \eqref{2.8} \,\, \text{holds for
all} \,\,t\in[1,\fz)\,\, \text{and}\,\,s\in(0,\fz)\}\noz
\end{eqnarray}
and
\begin{eqnarray}\label{2.10}
&&p_{\Phi}^-\equiv\sup\{p\in(0,\fz):\ \text{there exists a
positive constant}\,\,C\,\,\\
&&\hspace{4.8 em}\text{such that}\,\,\eqref{2.8} \
\text{holds for all}\ t\in(0,1)\ \text{and}\ s\in(0,\fz)\}.\noz
\end{eqnarray}
It is easy to see that $p_{\Phi}^-\le p_{\Phi}^{+}$
for all $\Phi$. In what follows, $p_{\Phi}^-$ and
$p_{\Phi}^{+}$ are respectively called the {\it critical
lower type index} and the {\it critical upper type index} of $\Phi$.

Throughout the whole paper, we always assume that $\Phi$ satisfies
the following \emph{assumption}.

\medskip

\noindent{\bf Assumption $(\Phi)$.} {Let $\Phi$ be a
positive, continuous, strictly increasing function on $(0,\fz)$
which is of critical lower type $p_{\Phi}^-
\in(0,1]$. Also assume that $\Phi$ is concave.}

\begin{remark}\rm\label{r2.2}
(i) Recall that the function $\Phi$ is called of \emph{strictly lower type} $p$
if \eqref{2.8} holds with $C\ev1$ for all $t\in(0,1)$ and $s\in(0,\fz)$.
Then the \emph{strictly critical lower type index}  $p_{\Phi}$
of $\Phi$ is defined by
\begin{equation*}
p_{\Phi}\equiv\sup\{p\in(0,\fz):\,\Phi(st)\le t^p\Phi(s) \
\text{holds for all}\ t\in(0,1)\ \text{and}\ s\in(0,\fz)\}.
\end{equation*}
Obviously, $p_\Phi\le p_\Phi^-\le p_\Phi^+$. Moreover, it was proved in
\cite[Remark 2.1]{jy10} that $\Phi$ is also of strictly lower type
$p_\Phi$. In other words, $p_\Phi$ is \emph{attainable}.

However, $p_\Phi^-$ and $p_\Phi^+$ may not be attainable.
For example, for $p\in(0,1]$, if $\Phi(t)\equiv t^p$ for all $t\in (0,\fz)$,
then $\Phi$ satisfies Assumption $(\Phi)$ and $p_\Phi=p_\Phi^-=p_\Phi^+=p$;
for $p\in[1/2,1]$, if $\Phi(t)\equiv t^p/\ln(e+t)$ for all $t\in (0,\fz)$,
then $\Phi$ satisfies Assumption $(\Phi)$ and
$p_\Phi^-=p=p_\Phi^+$, $p_\Phi^-$ is not attainable,
but $p_\Phi^+$ is attainable;
for $p\in(0,1/2]$, if $\Phi(t)\equiv t^p\ln(e+t)$ for all $t\in (0,\fz)$, then
then $\Phi$ satisfies Assumption $(\Phi)$ and
$p_\Phi^-=p=p_\Phi^+$, $p_\Phi^-$ is attainable,
but $p_\Phi^+$ is not attainable.

(ii) We observe that, via the Aoki-Rolewicz theorem in \cite{ao,Ro57},
all results in \cite{al11,jy10,jy10a,jy} are still
true if the \emph{assumptions on $p_\Phi$
are replaced by the same assumptions on $p_\Phi^-$}.
\end{remark}

Notice that if $\Phi$ satisfies Assumption $(\Phi)$, then
$\Phi(0)=0$. For any positive function $\wz\Phi$ of critical lower
type $p_{\wz\Phi}^-$, if we set $\Phi(t)\equiv\int_0^t
\frac{\wz\Phi(s)}{s}\,ds$ for $t\in[0,\fz)$, then by
\cite[Proposition 3.1]{vi87}, $\Phi$ is equivalent to $\wz\Phi$,
namely, there exists a positive constant $C$ such that
$C^{-1}\wz\Phi(t)\le\Phi(t)\le C\wz\Phi(t)$ for all $t\in[0,\fz)$;
moreover, $\Phi$ is a positive, strictly increasing, concave and
continuous function of critical lower type $p_{\wz\Phi}^-$. Notice
that all our results of this paper are invariant on equivalent Orlicz
functions. From this, we deduce that {\it all results with $\Phi$ as
in Assumption $(\Phi)$ also hold for all positive functions
$\wz\Phi$ of the same critical lower type $p_\Phi^-$ as
$\Phi$}.

Let $\Phi$ satisfy Assumption $(\Phi)$. A
measurable function $f$ on $\cx$ is said to be in the {\it space
$L^{\Phi}(\cx)$} if $\int_{\cx}\Phi(|f(x)|)\,d\mu(x)<\fz$. Moreover, for
any $f\in L^{\Phi}(\cx)$, define
$$\|f\|_{L^{\Phi}(\cx)}\equiv\inf\left\{\lz\in(0,\fz):
\ \int_{\cx}\Phi\lf(\frac{|f(x)|}{\lz}\r)\,d\mu(x)\le1\right\}.$$

Since $\Phi$ is strictly increasing, we define the function $\rho(t)$
on $(0,\fz)$ by
\begin{equation}\label{2.11}
\rho(t)\ev \frac{t^{-1}}{\Phi^{-1}(t^{-1})}
\end{equation}
for all $t\in (0,\fz)$, where $\Phi^{-1}$ is the \emph{inverse function} of
$\Phi$. Then the types of $\Phi$  and $\rho$ have the following relation.
\emph{If $0<p_0\le p_1\le1$ and $\Phi$ is an increasing function, then $\Phi$
is of type $(p_0,p_1)$ if and only if $\rho$ is of type
$(p_1^{-1}-1,p_0^{-1}-1);$} see \cite{vi87} for its proof.

\section{The Space $\vmo${\label{s3}}}

\hskip\parindent In this section, we introduce the generalized
vanishing mean oscillation spaces associated with $L$. Throughout
this section, we \emph{always assume} that $L$ satisfies Assumptions
$(L)_1$ and $(L)_2$.

We first recall the
notion of tent spaces in \cite{r07}, which when $\cx\ev\rn$ were
first introduced by Coifman, Meyer and Stein \cite{cms85}.

For any $\nu>0$ and $x\in\cx$, let
$\bgz_\nu(x)\equiv\{(y,t)\in\cx\times (0,\fz):\,d(x,y)<\nu t\}$
denote the {\it cone of aperture $\nu$ with vertex} $x\in\cx$. For
any closed set $F$ of $\cx$, denote by $\mr_\nu F$ the {\it union of
all cones with vertices in $F$}, namely, $\mr_\nu
F\equiv\bigcup_{x\in F}\bgz_\nu(x)$; and for any open set $O$ in
$\cx$, denote the {\it tent over $O$} by $T_\nu(O)$, which is
defined as $T_\nu(O)\ev[\mr_\nu(O^\com)]^\com$. It is easy to see
that $T_\nu(O)=\{(x,t)\in\cx\times(0,\fz):\,d(x,O^\com)\ge\nu t\}$.
In what follows, we denote $\mr_1(F)$, $\bgz_1(x)$ and $T_1(O)$
simply by $\mr(F)$, $\bgz(x)$ and $\wh O$, respectively.

For all measurable functions $g$ on $\xt$ and $x\in\cx$, define
$$\ca_\nu(g)(x)\ev\lf(\iint_{\bgz_\nu(x)}|g(y,t)|^2\,\frac{d\mu(y)}
{V(x,t)}\,\frac{dt}t\r)^{1/2}$$
and
$$\cro(g)(x)\ev\sup_{B\ni x}\frac{1}{\rho(\mu(B))}\lf(\frac{1}{\mu(B)}
\iint_{\wh B}|g(y,t)|^2\,\frac{d\mu(y)\,dt}t\r)^{1/2},$$
where the supremum is taken over all balls $B$ containing $x$.
We denote $\ca_1(g)$ simply by $\ca(g)$.

Recall that for $p\in(0,\fz)$, the {\it tent space $T_2^p(\cx)$} is
defined to be the space of all measurable functions $g$ on
$\xt$ such that $\|g\|_{T_2^p(\cx)}\ev\|\ca(g)\|_\lp<\fz$, which
when $\cx\equiv\rn$ was introduced by Coifman, Meyer and Stein
\cite{cms85} and when $\cx$ is a space of homogeneous type by Russ
in \cite{r07}. Let $\Phi$ satisfy Assumption $(\Phi)$. In what
follows, we denote by $\tx$ the {\it space of all measurable
functions $g$ on $\xt$ such that $\ca(g)\in L^\Phi(\cx)$}, and for
any $g\in\tx$, define its {\it norm} by
$$\|g\|_{\tx}\ev\|\ca(g)\|_\lx=\inf\lf\{\lz>0:\,\int_\cx\Phi\lf(\frac
{\ca(g)(x)}\lz\r)\,d\mu(x)\le1\r\};$$
the {\it space} $\txz$ is defined to be the space of
all measurable functions $g$ on $\xt$
satisfying $\|g\|_\txz\ev\|\cro(g)\|_{L^\fz(\cx)}<\fz$.

Recall that a function $a$ on $\xt$ is called a {\it $\tx$-atom} if

{\rm (i)} there exists a ball $B\st\cx$ such that $\supp a\st\wh B$;

{\rm(ii)} $\iint_{\wh B}|a(x,t)|^2\,\frac{d\mu(x)\,dt}t\le[\mu(B)]^{-1}
[\rho(\mu(B))]^{-2}$.

Since $\Phi$ is concave, from Jensen's inequality and
H\"older's inequality, we deduce that for all $\tx$-atoms
$a$, $\|a\|_\tx\le1$; see \cite{jy} for the details.
Moreover, the following atomic decomposition for elements in
$\tx$ is just \cite[Theorem 3.1]{jy}.

\begin{lemma}\label{l3.1}
Let $\Phi$ satisfy Assumption $(\Phi)$. Then for any $f\in\tx$,
there exist $\tx$-atoms $\{a_j\}_{j=1}^\fz$ and
$\{\lz_j\}_{j=1}^\fz\st\cc$ such that for almost every
$(x,t)\in\xt$,
\begin{equation}\label{3.1}
f(x,t)=\sum_{j=1}^\fz\lz_j a_j(x,t),
\end{equation}
and the series converges in $\tx$. Moreover, there exists
a positive constant $C$ such that for all $f\in\tx$,
\begin{equation}\label{3.2}
\blz(\{\lz_j a_j\}_{j=1}^\fz)\ev\inf\lf\{\lz>0:\,\sum_{j=1}^\fz \mu(B_j)\Phi
\lf(\frac{|\lz_j|}{\lz \mu(B_j)\rho(\mu(B_j))}\r)\le1\r\}\le C\|f\|_\tx,
\end{equation}
where $\wh B_j$ appears as the support of $a_j$.
\end{lemma}

\begin{definition}\rm\label{d3.1}
Let $L$ satisfy Assumptions $(L)_1$ and $(L)_2$, $\Phi$ satisfy
Assumption $(\Phi)$, $\rho$ be as in \eqref{2.11},
$M\in\nn$, $\ez\in(0,\fz)$ and $B$
be a ball. A function $\bz\in L^2(\cx)$ is called a
\emph{$(\Phi,\,M,\,\ez)_L$-molecule adapted to the ball $B$} if there
exists a function $b\in\cd(L^M)$ such that

{\rm (i)} $\bz=L^Mb$;

{\rm(ii)} For every $k\in\{0,1,\cdots,M\}$ and $j\in\zz_+$, there holds
$$\|(r_B^2L)^kb\|_{L^2(U_j(B))}\le r_B^{2M}2^{-j\ez}[\mu(2^jB)]^{-1/2}
[\rho(\mu(2^jB))]^{-1},$$
where $U_j(B)$ for $j\in\zz_+$ is as in \eqref{2.5}.
\end{definition}

Let $\phi=L^M\nu$ be a function in $L^2(\cx)$, where
$\nu\in\cd(L^M)$. Following \cite{hlmmy, hm09}, for $\ez>0$,
$M\in\nn$ and a fixed $x_0\in\cx$, we introduce the {\it space}
\begin{equation}\label{3.3}
\cm^{M,\ez}_\Phi(L)\ev\lf\{\phi=L^M\nu\in L^2(\cx):\ \|\phi\|_
{\cm^{M,\ez}_\Phi(L)}<\fz\r\},
\end{equation}
where
$$\|\phi\|_{\cm^{M,\ez}_\Phi(L)}\ev\sup_{j\in\zz_+}\lf\{2^{j\ez}[V(x_0,
2^j)]^{1/2}\rho(V(x_0,2^j))\sum^M_{k=0}\|L^k\nu\|_{L^2(U_j(B(x_0,1)))}\r\};$$
see also \cite{al11}.

Notice that if $\phi\in\ml$ for some $\ez>0$ with norm $1$, then
$\phi$ is a $\pme$-molecule adapted to the ball
$B(x_0,1)$. Conversely, if $\bz$ is a $\pme$-molecule
adapted to any ball, then $\bz\in\ml$.

Let $A_t$ denote either $(I+t^2L)^{-1}$ or $e^{-t^2L}$ and
$A_t^*$ either $(I+t^2L^*)^{-1}$ or $e^{-t^2L^*}$.
For any $f\in(\mlx)^*$,
the {\it dual space } of $\mlx$, we claim that
{\it $(I-A_t)^Mf\in L^2_{\rm loc}(\cx)$
in the sense of distributions}.
Indeed, for any ball $B$, if $\psi\in L^2(B)$,
then it follows form the Davies-Gaffney estimate \eqref{2.7}
and Remark \ref{r2.1} that
$(I-A_t^*)^M\psi\in\mlx$ for every $\ez>0$. Thus, there exists
a non-negative constant $C(t,r_B,\dist(B,x_0))$, depending
on $t$, $r_B$ and $\dist(B,x_0)$, such that for all $\psi\in L^2(B)$,
$$|\la(I-A_t)^Mf,\psi\ra|\equiv|\la f,(I-A_t^*)^M\psi\ra|\le C(t,r_B,
\dist(B,x_0))\|f\|_{(\mlx)^*}\|\psi\|_{L^2(B)},$$ which implies that
$(I-A_t)^Mf\in L^2_{\rm loc}(\cx)$ in the sense of distributions.

Finally, for any $M\in\nn$, define
\begin{equation}\label{3.4}
\cm^M_{\Phi,L}(\cx)\equiv\bigcap_{\ez>n(1/p_\Phi^--1/p_\Phi^+)}
(\cm_\Phi^{M,\ez}(L^*))^*,
\end{equation}
where $p_\Phi^+$ and $p_\Phi^-$ are, respectively, as in \eqref{2.9}
and \eqref{2.10}.

\begin{definition}\rm\label{d3.2}
Let $L$, $\Phi$ and $\rho$ be as in Definition \ref{d3.1}
and $M>\frac n2(\frac1{p_\Phi^-}-\frac12)$.
A function $f\in\cm^M_{\Phi,L}(\cx)$ is said to be in the \emph{space} $\bmol$ if
$$\|f\|_{\bmol}\equiv\sup_{B\subset\cx}\frac1{\rho(\mu(B))}\lf[\frac1{\mu(B)}
\int_B|(I-e^{-r_B^2L})^Mf(x)|^2\,d\mu(x)\r]^{1/2}<\fz,$$
where the supremum is taken over all balls $B$ of $\cx$.
\end{definition}

Now, let us recall some notions on the Orlicz-Hardy spaces associated with
$L$. For all $f\in L^2(\cx)$ and $x\in\cx$, define
$$\cs_Lf(x)\ev\lf(\iint_{\bgz(x)}\lf|\tl f(y)\r|^2\dyt\r)^{1/2}.$$
The {\it Orlicz-Hardy space} $\hx$ is defined to be the completion of the set $\{f\in
L^2(\cx):\,\cs_Lf\in\lx\}$ with respect to the quasi-norm $\|f\|_\hx
\ev\|\cs_Lf\|_\lx$.

The Orlicz-Hardy space $\hx$ was introduced and studied in \cite{al11}
(see also \cite{jy}).
If $\Phi(t)\ev t^p$ for $p\in(0,1]$ and all $t\in(0,\fz)$, then the space $\hx$
coincides with the Hardy space $H_L^p(\cx)$, which was introduced
and studied by Duong and Li \cite{dl}.

Let the \emph{space} $\hmfl$ denote the spaces of finite
linear combinations of $\pme$-molecules.
By \cite[Corollary 3.8]{al11}, we obtain that $\hmfl$ is dense
in $\hx$; see also \cite[Corollary 4.2]{jy}.

In what follows, for $M\in\nn$, let $C(M)$ be the
\emph{positive constant} such that
\begin{equation}\label{3.5}
C(M)\int_0^\fz t^{2(M+1)}e^{-2t^2}\dt=1.
\end{equation}

Recall that a variant of the following representation of finite linear
combinations of molecules was gives by \cite[Theorem 3.15]{al11} without
a detailed proof. The following Theorem \ref{t3.1} gives more \emph{accurate
ranges} of $\ez$ and $M$, comparing with \cite[Theorem 3.15]{al11}.

\begin{theorem}\label{t3.1}
Let $L$, $\Phi$ and $M$ be as in Definition \ref{d3.2}, and $\ez\in(0,M-\mz)$.
Assume that $f=\sum_{i=0}^N\lz_ia_i$, where $N\in\nn$, $\{a_i\}_{i=0}^N$ is a
family of $(\Phi,2M,\ez)_L$-molecules, $\{\lz_i\}_{i=0}^N\subset\cc$
and $\sum_{i=0}^N|\lz_i|<\fz$.
Then there exists a representation of $f=\sum_{i=0}^{2N}\mu_im_i$, where
$\{m_i\}_{i=1}^{2N}$ are $\pme$-molecules, $\{\mu_i\}_{i=0}^{2N}\subset\cc$ and
$\sum_{i=0}^{2N}|\mu_i|\le C\|f\|_\hx,$
where $C$ is a positive constant, depending only on $\cx,L,M,\ez$ and $n$.
\end{theorem}

\begin{proof} Throughout this proof, we choose
$\wz p_\Phi\in(0, p_\Phi^-)$ such that $M>\mzx$ and $\ez\in(0,M-\mzx)$.
Therefore, $\Phi$ is of \emph{lower type} $\wz p_\Phi$ and hence $\rho$
of \emph{upper type} $1/\wz p_\Phi-1$.

Since $\{a_i\}_{i=0}^N$ is a family of $(\Phi,2M,\ez)_L$-molecules,
by definition there exist a family $\{b_i\}_{i=0}^N$ of functions
and a family $\{B_i\}_{i=0}^N$ of balls such that for every $i\in
\{0,1,\cdots,N\}$, $a_i=L^{2M}b_i$ satisfies
Definition \ref{d3.1}(ii). Fix a point $x_0\in\cx$.
Let $\wz C(M)\ev\frac{2C(M)}{M+1}$, where $C(M)$ is as in \eqref{3.5}.
Then $\wz C(M)\int_0^\fz t^{2(M+2)}e^{-2t^2}\dt=1$.
By this and the $L^2$-functional calculus, for
$f=\sum_{i=0}^N\lz_ia_i\in L^2(\cx)$, we have
\begin{eqnarray*}
f&&=\wz C(M)\int_0^\fz (t^2L)^{M+2}e^{-2t^2L}f\dt\\
&&=\wz C(M)\int_{K_1}^\fz (t^2L)^{M+2}e^{-2t^2L}f\dt
+\wz C(M)\int_0^{K_1}\cdots\ev f_1+f_2,
\end{eqnarray*}
where $K_1$ is a \emph{positive constant} which is determined later.

Let us start with the term $f_1$. Set $\mu\ev N^{-1}\|f\|_\hx$.
Substituting $f=\sum_{i=0}^N\lz_ia_i$ into $f_1$, we have
\begin{equation*}
f_1=\wz C(M)\sum_{i=0}^N\lz_i\int_{K_1}^\fz (t^2L)^{M+2}e^{-2t^2L}a_i\dt
=\sum_{i=0}^N\mu_im_{i,K_1},
\end{equation*}
where $\mu_i\ev\wz C(M)\mu$, $m_{i,K_1}\ev L^Mf_{i,K_1}$, and
$$f_{i,K_1}\ev\mu^{-1}\lz_i\int_{K_1}^\fz t^{2(M+2)}L^2
e^{-2t^2L}a_i\dt.$$
Then, obviously, $\sum_{i=0}^N|\mu_i|=\sum_{i=0}^N\mu_i=C(M)\|f\|_\hx$.
We now claim that
for an appropriate choice of $K_1$ and $i\in\{0,1,\cdots,N\}$,
$m_{i,K_1}$ is a $\pme$-molecule adapted to the
ball $B_i$. Observe that $a_i=L^{2M}b_i$, for $i\in\{0,1,\cdots,N\}$.
By Minkowski's inequality, for $k\in\{0,1,\cdots,M\}$,
$i\in\{0,1,\cdots,N\}$ and $j\in \zz_+$,
\begin{eqnarray*}
&&\lf\|(r_{B_i}^2L)^kf_{i,K_1}\r\|_{L^2(U_j(B_i))}\\
&&\hs\le\mu^{-1}|\lz_i|\int_{K_1}^\fz t^{-2M}
\lf\|(t^2L)^{2(M+1)}e^{-2t^2L}(r_{B_i}^2L)^kb_i\r\|_{L^2(U_j(B_i))}\dt\\
&&\hs\le\mu^{-1}|\lz_i|\sum_{l=0}^\fz\int_{K_1}^\fz t^{-2M}
\lf\|(t^2L)^{2(M+1)}e^{-2t^2L}\lf(\chi_{U_l(B_i)}
\lf[(r_{B_i}^2L)^kb_i\r]\r)\r\|_{L^2(U_j(B_i))}\dt\\
&&\hs\ev\mu^{-1}|\lz_i|\sum_{l=0}^\fz\hl,
\end{eqnarray*}
where $U_l(B_i)$ for $l\in\zz_+$ is as in \eqref{2.5}.
When $l<j-1$, by Lemma \ref{l2.2},
$\mu(2^jB_i)\ls2^{n(j-l)}\mu(2^lB_i)$,
$\rho(\mu(2^jB_i))\ls2^{n(j-l)(1/\wz p_\Phi-1)}\rho(\mu(2^lB_i))$
and Definition \ref{d3.1}(ii),
we conclude that
\begin{eqnarray*}
\hl&&\ls\int_{K_1}^\fz t^{-2M}\lf\|(r_{B_i}^2L)^kb_i\r\|_{L^2(U_l(B_i))}
\lf(\frac t{2^jr_{B_i}}\r)^{\ez+n(1/\wz p_\Phi-1/2)}\dt\\
&&\ls\int_{K_1}^\fz
t^{-2M}r_{B_i}^{4M}2^{-l\ez}[\mu(2^lB_i)]^{-1/2}[\rho(\mu(2^lB_i))]^{-1}
\lf(\frac t{2^jr_{B_i}}\r)^{\ez+n(1/\wz p_\Phi-1/2)}\dt\\
&&\ls
r_{B_i}^{2M}2^{-j\ez}[\mu(2^jB_i)]^{-1/2}[\rho(\mu(2^jB_i))]^{-1}2^{-l(\ez+\mzx)}
\lf(\frac {r_{B_i}}{K_1}\r)^{2[M-\frac\ez2-\mzx]}.
\end{eqnarray*}
When $l\in\{j-1,j,j+1\}$, from Lemma \ref{l2.2}
and Definition \ref{d3.1}(ii), it follows that
\begin{eqnarray*}
\hl&&\ls\int_{K_1}^\fz t^{-2M}\lf\|(r_{B_i}^2L)^kb_i\r\|_{L^2(U_j(B_i))}\dt\\
&&\ls r_{B_i}^{2M}2^{-j\ez}[\mu(2^jB_i)]^{-1/2}[\rho(\mu(2^jB_i))]^{-1}
\lf(\frac {r_{B_i}}{K_1}\r)^{2M}.
\end{eqnarray*}
When $l>j+1$, by Lemma \ref{l2.2},
$\mu(2^jB_i)\ls \mu(2^lB_i)$, $\rho(\mu(2^jB_i))\ls\rho(\mu(2^lB_i))$
and Definition \ref{d3.1}(ii), we obtain
\begin{eqnarray*}
\hl&&\ls\int_{K_1}^\fz t^{-2M}\lf\|(r_{B_i}^2L)^kb_i\r\|_{L^2(U_l(B_i))}
\lf(\frac t{2^lr_{B_i}}\r)^{\ez}\dt\\
&&\ls r_{B_i}^{2M}2^{-j\ez}[\mu(2^jB_i)]^{-1/2}[\rho(\mu(2^jB_i))]^{-1}2^{-l\ez}
\lf(\frac {r_{B_i}}{K_1}\r)^{2M-\ez}.
\end{eqnarray*}
Combining these estimates, by choosing $K_1>\max\{r_{B_1},\cdots,r_{B_N}\}$,
we further conclude that there exists a positive constant $\wz C$,
independent of $i$, such that
\begin{eqnarray*}
&&\lf\|(r_{B_i}^2L)^kf_{i,K_1}\r\|_{L^2(U_j(B_i))}\\
&&\hs\le \wz Cr_{B_i}^{2M}2^{-j\ez}[\mu(2^jB_i)]^{-1/2}[\rho(\mu(2^jB_i))]^{-1}
\mu^{-1}|\lz_i|\lf(\frac {r_{B_i}}{K_1}\r)^{2[M-\frac\ez2-\mzx]}.
\end{eqnarray*}
Then, by choosing
$$K_1\ev
\max_{0\le i\le N}\lf\{r_{B_i}\lf[\wz C\mu^{-1}\max_{0\le i\le N}|\lz_i|\r]
^{\frac1{2[M-\frac\ez2-\mzx]}}\r\},$$
we see that for $i\in\{0,1,\cdots,N\}$, $m_{i,K_1}$ is a $\pme$-molecule
adapted to the ball $B_i$, which shows the claim.

We now consider the term $f_2$. Set $\mu\ev N^{-1}\|f\|_\hx$.
Substituting $f=\sum_{i=0}^N\lz_ia_i$ into $f_2$, we have
\begin{equation*}
f_2=\wz C(M)\sum_{i=0}^N\lz_i\int_0^{K_1}(t^2L)^{M+1}e^{-t^2L}(\tl a_i)\dt
=\sum_{i=0}^N\mu_im_{i,K_1},
\end{equation*}
where $\mu_i\ev C(M)\mu$, $m_{i,K_1}\ev L^Mf_{i,K_1}$, and
$$f_{i,K_1}\ev\mu^{-1}\lz_i\int_0^{K_1} t^{2(M+1)}L
e^{-t^2L}(\tl a_i)\dt.$$
Then, obviously, $\sum_{i=0}^N|\mu_i|=\sum_{i=0}^N\mu_i=C(M)\|f\|_\hx$.
We now claim that for $K_1$ as above and $i\in\{0,1,\cdots,N\}$,
$m_{i,K_1}$ is a $\pme$-molecule adapted to the
ball $2^{K_0}B_i$, where $K_0\in(0,\fz)$ is determined later.
To show the claim, for $i\in\{0,1,\cdots,N\}$ and $j\in \zz_+$,
set
$\boz_{j,K_0}\ev 2^{j+K_0+2}B_i\bh2^{j+K_0-2}B_i$
and write
\begin{eqnarray*}
f_{i,K_1}=&&\mu^{-1}\lz_i\int_0^{K_1} t^{2(M+1)}L
e^{-t^2L}\lf([\tl a_i]\chi_{\boz_{j,K_0}}\r)\dt\\
&&+\mu^{-1}\lz_i\int_0^{K_1} t^{2(M+1)}L
e^{-t^2L}\lf([\tl a_i]\chi_{\boz_{j,K_0}^\com}\r)\dt
\ev g_{i,K_1,K_0}+h_{i,K_1,K_0}.
\end{eqnarray*}
Then, by Minkowski's inequality, for $k\in\{0,1,\cdots,M\}$,
$i\in\{0,1,\cdots,N\}$ and $j\in \zz_+$,
\begin{eqnarray*}
&&\lf\|(2^{2K_0}r_{B_i}^2L)^kg_{i,K_1,K_0}\r\|_{L^2(U_j(2^{K_0}B_i))}\\
&&\hs\le\mu^{-1}|\lz_i|r_{B_i}^{2M}\lf\|\int_0^{K_1}
\lf(\frac t{r_{B_i}}\r)^{2M-2k}2^{2kK_0}\r.\\
&&\hs\hs\times\lf.
(t^2L)^{k+1}e^{-t^2L}\lf(\lf[\tl a_i\r]\chi_{\boz_{j,K_0}}\r)
\dt\r\|_{L^2(U_j(2^{K_0}B_i))}\\
&&\hs\le
\mu^{-1}|\lz_i|\sum_{l=0}^\fz\int_0^{K_1}\lf(\frac t{r_{B_i}}\r)^{2M-2k}2^{2kK_0}
\lf\|\chi_{U_l(2^{K_0}B_i)}
\tl a_i\r\|_{L^2(\boz_{j,K_0})}\dt\\
&&\hs\ev\mu^{-1}|\lz_i|\sum_{l=0}^\fz\hl.
\end{eqnarray*}
When $l<j-2$, from Lemma \ref{l2.2},
$\mu(2^{j+K_0}B_i)\ls2^{n(j-l)}\mu(2^{l+K_0}B_i)$,
$\rho(\mu(2^{j+K_0}B_i))\ls2^{n(j-l)
(1/\wz p_\Phi-1)}\rho(\mu(2^{l+K_0}B_i))$ and Definition \ref{d3.1}(ii),
it follows that
\begin{eqnarray*}
\hl&&\ls\int_0^{K_1}
\lf(\frac t{r_{B_i}}\r)^{2M-2k}2^{2kK_0}\lf\|a_i\r\|_{L^2(U_l(2^{K_0}B_i))}
\lf(\frac t{2^{j+K_0}r_{B_i}}\r)^{\ez+n(1/\wz p_\Phi-1/2)}\dt\\
&&\ls\int_0^{K_1} \lf(\frac t{r_{B_i}}\r)^{2M-2k}2^{2kK_0}r_{B_i}^{4M}
2^{-(l+K_0)\ez}[\mu(2^{l+K_0}B_i)]^{-1/2}[\rho(\mu(2^{l+K_0}B_i))]^{-1}\\
&&\hs\times\lf(\frac t{2^{j+K_0}r_{B_i}}\r)^{\ez+n(1/\wz p_\Phi-1/2)}\dt\\
&&\ls (2^{K_0}r_{B_i})^{2M}2^{-j\ez}[\mu(2^{j+K_0}B_i)]^{-1/2}
[\rho(\mu(2^{j+K_0}B_i))]^{-1}2^{-l[\ez+\mzx]}\\
&&\hs\times2^{-2K_0[M-k+\ez+\mzx]}
K_1^{2M-2k+\ez+n(1/\wz p_\Phi-1/2)}r_{B_i}^{2M+2k-\ez-n(1/\wz p_\Phi-1/2)}.
\end{eqnarray*}
When $l\in\{j-2,\cdots,j+2\}$, by Lemma \ref{l2.2} and Definition \ref{d3.1}(ii),
we see that
\begin{eqnarray*}
\hl&&\ls\int_0^{K_1} \lf(\frac t{r_{B_i}}\r)^{2M-2k}2^{2kK_0}
\lf\|a_i\r\|_{L^2(U_j(2^{K_0}B_i))}\dt\\
&&\ls (2^{K_0}r_{B_i})^{2M}2^{-j\ez}[\mu(2^{j+K_0}B_i)]^{-1/2}
[\rho(\mu(2^{j+K_0}B_i))]^{-1}2^{-2K_0(M-k+\ez/2)}
K_1^{2M-2k}r_{B_i}^{2M+2k}.
\end{eqnarray*}
When $l>j+2$, from Lemma \ref{l2.2},
$\mu(2^jB_i)\ls \mu(2^lB_i)$, $\rho(\mu(2^{j+K_0}B_i))\ls\rho(\mu(2^{l+K_0}B_i))$
and Definition \ref{d3.1}(ii), we infer that
\begin{eqnarray*}
\hl&&\ls\int_0^{K_1}
\lf(\frac t{r_{B_i}}\r)^{2M-2k}2^{2kK_0}\lf\|a_i\r\|_{L^2(U_l(2^{K_0}B_i))}
\lf(\frac t{2^{l+K_0}r_{B_i}}\r)^{\ez}\dt\\
&&\ls\int_0^{K_1} \lf(\frac t{r_{B_i}}\r)^{2M-2k}2^{2kK_0}r_{B_i}^{4M}
2^{-(l+K_0)\ez}[\mu(2^{l+K_0}B_i)]^{-1/2}[\rho(\mu(2^{l+K_0}B_i))]^{-1}\\
&&\hs\times\lf(\frac t{2^{l+K_0}r_{B_i}}\r)^{\ez}\dt\\
&&\ls (2^{K_0}r_{B_i})^{2M}2^{-j\ez}[\mu(2^{j+K_0}B_i)]^{-1/2}
[\rho(\mu(2^{j+K_0}B_i))]^{-1}2^{-l\ez}\\
&&\hs\times2^{-2K_0(M-k+\ez)}K_1^{2M-2k+\ez}r_{B_i}^{2M+2k-\ez}.
\end{eqnarray*}

Then we estimate $h_{i,K_1,K_0}$. By Minkowski's inequality
and Definition \ref{d3.1}(ii), for $k\in\{0,1,\cdots,M\}$,
$i\in\{0,1,\cdots,N\}$ and $j\in \zz_+$, we conclude that
\begin{eqnarray*}
&&\lf\|(2^{2K_0}r_{B_i}^2L)^kh_{i,K_1,K_0}\r\|_{L^2(U_j(2^{K_0}B_i))}\\
&&\hs\le\mu^{-1}|\lz_i|r_{B_i}^{2M}\lf\|\int_0^{K_1}
\lf(\frac t{r_{B_i}}\r)^{2M-2k}2^{2kK_0}\r.\\
&&\hs\hs\times\lf.
(t^2L)^{k+1}e^{-t^2L}\lf(\lf[\tl a_i\r]
\chi_{\boz_{j,K_0}^\com}\r)\dt\r\|_{L^2(U_j(2^{K_0}B_i))}\\
&&\hs\le\mu^{-1}|\lz_i|\int_0^{K_1}\lf(\frac t{r_{B_i}}\r)^{2M-2k}2^{2kK_0}
\lf(\frac t{2^{j+K_0}r_{B_i}}\r)^{\ez+n(1/\wz p_\Phi-1/2)}
\lf\|\tl a_i\r\|_{L^2(\cx)}\dt\\
&&\hs\ls (2^{K_0}r_{B_i})^{2M}2^{-j\ez}[\mu(2^{j+K_0}B_i)]^{-1/2}
[\rho(\mu(2^{j+K_0}B_i))]^{-1}\\
&&\hs\hs\times2^{-2K_0[M-k+\ez+\mzx]}
K_1^{2M-2k+\ez+n(1/\wz p_\Phi-1/2)}r_{B_i}^{2M+2k-\ez-n(1/\wz p_\Phi-1/2)}.
\end{eqnarray*}

Combining these estimates, by choosing $K_1>\max\{r_{B_1},\cdots,r_{B_N}\}$,
we further see that
\begin{eqnarray*}
\lf\|(2^{2K_0}r_{B_i}^2L)^kf_{i,K_1}\r\|_{L^2(U_j(2^{K_0}B_i))}
&&\ls(2^{K_0}r_{B_i})^{2M}2^{-j\ez}[\mu(2^{j+K_0}B_i)]^{-1/2}
[\rho(\mu(2^{j+K_0}B_i))]^{-1}\\
&&\hs\hs\times2^{-2K_0(M-k+\ez/2)}K_1^{2M-2k+\ez+\mzx}r_{B_i}^{2M+2k}.
\end{eqnarray*}
Then, by choosing
$$K_0\ev\max_{0\le k\le M}\lf(\frac {\ln\Big(K_1^{2M-2k+\ez+\mzx}
{\max_{0\le i\le N}}\{r_{B_i}^{2M+2k}\}\Big)}{2\ln2(M-k+\ez/2)}\r),$$
we conclude that for $i\in\{0,1,\cdots,N\}$, $m_{i,K_1}$ is
a $\pme$-molecule adapted to the ball $2^{K_0}B_i$, which shows
the claim, and hence completes the proof of Theorem \ref{t3.1}.
\end{proof}

\begin{remark}\rm\label{r3.1} We point out that the proof of
Theorem \ref{t3.1} also works for \cite[Theorem 5.4]{hlmmy}.
Moreover, due to the lack of the support of molecules,
we show that $m_{i, K_1}$ for $i\in\{1,\cdots, N\}$
is a $\pme$-molecule adapted to
the ball $2^{K_0}B_i$, instead of $B_i$ as in the proof of
\cite[Theorem 5.4]{hlmmy}, which also simplifies the proof
of \cite[Theorem 5.4]{hlmmy}.
\end{remark}

By Theorem \ref{t3.1}, the argument same as the proofs
of \cite[Theorems 3.13 and 3.16]{al11}, we obtain the following
dual theorem. We omit the details.

\begin{theorem}\label{t3.2}
Let $L$, $\Phi$, $\rho$ and $M$ be as in Definition \ref{d3.2}. Then for any
function $f\in\bmol$, the linear functional $\ell$, defined by
$\ell(g)\ev\la f,g\ra$
initially on $H_{\Phi,{\rm fin},L^*}^{\rm{mol},\ez,2\wz M}(\cx)$ with $\wz M>M$
and $\ez\in(0,\wz M-\mz)$, has a unique extension to $\hxx$ and, moreover,
$\|\ell\|_{(\hxx)^*}\le C\|f\|_\bmol$
for some nonnegative constant $C$ independent of $f$.

Conversely, for any $\ell\in(\hxx)^*$, there exists $f\in\bmol$
such that
$\ell(g)\ev\la f,g\ra$
for all $g\in H_{\Phi,fin,L^*}^{mol,\ez, M}(\cx)$ and $\|f\|_\bmol\le C
\|\ell\|_{(\hxx)^*}$, where $C$ is a nonnegative constant independent of $\ell$.
\end{theorem}

\begin{remark}\rm\label{r3.2}
(i) Theorem \ref{t3.1} is just \cite[Theorems 3.15]{al11} but with the ranges
of indices $M$ and $\ez$ replaced, respectively, by $M>\mz$ and $\ez\in(0,M-\mz)$.

(ii) By Theorem \ref{t3.2}, we see that for all $M>\mz$, the spaces $\bmol$ for
different $M$ coincide with equivalent norms; thus, in what follows,
we denote $\bmol$ simply by $\bmo$.
\end{remark}

The following two propositions are just \cite[Propositions 3.11 and 3.12]{al11}
(see also \cite[Propositions 4.4 and 4.5]{jy}).

\begin{proposition}\label{p3.1}
Let $L$, $\Phi$, $\rho$ and $M$ be as in Definition \ref{d3.2}. Then
$f\in\bmo$ if and only if $f\in\cm^M_{\Phi,L}(\cx)$ and
$$\sup_{B\subset\cx}\frac1{\rho(\mu(B))}\lf[\frac1{\mu(B)}
\int_B\lf|\lf[I-(I+r_B^2L)^{-1}\r]^Mf(x)\r|^2\,d\mu(x)\r]^{1/2}<\fz.$$
Moreover, the quantity appeared in the left-hand side of the above
formula is equivalent to $\|f\|_{\bmol}$.
\end{proposition}

\begin{proposition}\label{p3.2}
Let $L$, $\Phi$, $\rho$ and $M$ be as in Definition \ref{d3.2}. Then there
exists a positive constant $C$ such that for all $f\in\bmo$,
$$\sup_{B\subset\cx}\frac1{\rho(\mu(B))}\lf[\frac1{\mu(B)}\iint_{\wh B}
|(t^2L)^Me^{-t^2L}f(x)|^2\,\frac{d\mu(x)\,dt}t\r]^{1/2}\le C\|f\|_{\bmol}.$$
\end{proposition}

The following Proposition \ref{p3.3} and Lemma \ref{l3.2}
are a kind of Calder\'on reproducing formulae.

\begin{proposition}\label{p3.3}
Let $L$, $\Phi$, $\rho$ and $M$ be as in Definition \ref{d3.2},
$\ez, \ez_1\in(0,\fz)$
and $\wz M>M+\ez_1+\frac n4+\frac N2(\frac 1{p_\Phi^-}-1)$, where
$N$ is as in \eqref{2.4}. Fix $x_0\in\cx$. Assume that
$f\in\cm_{\Phi,L}^{M}(\cx)$ satisfies that
\begin{equation}\label{3.6}
\int_\cx
\frac{|(I-(I+L)^{-1})^Mf(x)|^2}{1+[d(x,x_0)]^{n+\ez_1+2N(1/p_\Phi^--1)}}
\,d\mu(x)<\fz.
\end{equation}
Then for all $\pml$-molecules $\az$,
$$\la f,\az\ra=C(M)\iint_\xt\tml f(x) \ov{\tlx\az(x)}\dxt,$$
where $C(M)$ is as in \eqref{3.5}.
\end{proposition}

\begin{proof}
For $R>\dz>0$, write
\begin{eqnarray*}
&&C(M)\int_\dz^R\int_\cx\tml f(x) \ov{\tlx\az(x)}\dxt\\
&&\hs=\lf\la f,C(M)\int_\dz^R (t^2L^*)^{M+1}e^{-2t^2L^*}\az\dt\r\ra\\
&&\hs=\la f,\az\ra-\lf\la f,\az-C(M)\int_\dz^R (t^2L^*)^{M+1}e^{-2t^2L^*}\az\dt\r\ra.
\end{eqnarray*}
Since $\az$ is a $\pml$-molecule, by Definition \ref{d3.1},
there exists $b\in L^2(\cx)$ such that $\az=(L^*)^{\wz M}b$.
Notice that
\begin{eqnarray*}
f&&=\lf[I-(I+L)^{-1}+(I+L)^{-1}\r]^Mf\\
&&=\sum_{k=0}^M\binom M k\lf[I-(I+L)^{-1}\r]^{M-k}(I+L)^{-k}f
=\sum_{k=0}^M\binom M k\lf[I-(I+L)^{-1}\r]^{M}L^{-k}f,
\end{eqnarray*}
where $\binom M k$ denotes the \emph{binomial coefficient},
which, together with $H_\fz$-functional calculus, further implies that
\begin{eqnarray*}
&&\lf\la f,\az-C(M)\int_\dz^R (t^2L^*)^{M+1}e^{-2t^2L^*}\az\dt\r\ra\\
&&\hs=\sum_{k=0}^M\binom M k
\lf\la \lf[I-(I+L)^{-1}\r]^{M}f,
L^{\wz M-k}b-C(M)\int_\dz^R (t^2L^*)^{M+1}e^{-2t^2L^*}(L^*)^{\wz M-k}b\dt\r\ra\\
&&\hs=\sum_{k=0}^M\binom M k
\lf\la \lf[I-(I+L)^{-1}\r]^{M}f,
C(M)\int_0^\dz (t^2L^*)^{M+1}e^{-2t^2L^*}(L^*)^{\wz M-k}b\dt\r\ra\\
&&\hs\hs+\sum_{k=0}^M\binom M k
\lf\la \lf[I-(I+L)^{-1}\r]^{M}f,
C(M)\int_R^\fz (t^2L^*)^{M+1}e^{-2t^2L^*}(L^*)^{\wz M-k}b\dt\r\ra\\
&&\hs\ev
\sum_{k=0}^M\binom M k(\mh+\mj).
\end{eqnarray*}

For $\mj$, by \eqref{3.6} and H\"older's inequality, we conclude that
\begin{eqnarray*}
|\mj|&&\ls
\lf\{\int_\cx\frac{|(I-(I+L)^{-1})^Mf(x)|^2}{1+[d(x,x_0)]^{n+\ez_1+2N(1/p_\Phi^--1)}}
\,d\mu(x)\r\}^{1/2}\\
&&\hs\times
\lf\{\int_\cx\lf|
\int_R^\fz(t^2L^*)^{M+\wz M-k+1}e^{-2t^2L^*}b(x)\frac1{t^{2(\wz M-k)+1}}\,dt\r|^2
\r.\\
&&\hs\times\lf(1+[d(x,x_0)]^{n+\ez_1+2N(1/p_\Phi^--1)}\r)\,d\mu(x)\Bigg\}^{1/2}\\
&&\ls
\int_R^\fz\lf\|(t^2L^*)^{M+\wz M-k+1}e^{-2t^2L^*}b
\lf(1+[d(\cdot,x_0)]^{n+\ez_1+2N(1/p_\Phi^--1)}\r)^{1/2}\r\|_{L^2(\cx)}\\
&&\hs\times\frac 1{t^{2(\wz M-k)+1}}\,dt.
\end{eqnarray*}
Let $B_0\ev B(x_0, 1)$. Notice that there exist $\wz N,\,d\in\nn$
such that for all $j\in\nn$, $j\ge \wz N$,
$$U_j(B_0)\st\bigcup_{i=-d}^{d}U_{j+i}(B),$$
where $B$ is the ball adapted to $\az$ and $U_j(B)$ for $j\in\zz_+$
is as in \eqref{2.5}.
By choosing $j_0\ge\wz N$, we conclude that
\begin{eqnarray*}
|\mj|
&&\ls
\int_R^\fz\lf\|(t^2L^*)^{M+\wz M-k+1}e^{-2t^2L^*}b\r.\\
&&\hs\lf.\times
(1+[d(\cdot,x_0)]^{n+\ez_1+2N(1/p_\Phi^--1)})^{1/2}\r\|_{L^2(2^{j_0}B_0)}
\frac1{t^{2(\wz M-k)+1}}\,dt\\
&&\hs+\sum_{j=j_0+1}^\fz
\int_R^\fz\lf\|(t^2L^*)^{M+\wz M-k+1}e^{-2t^2L^*}b\r.\\
&&\hs\lf.\times
(1+[d(\cdot,x_0)]^{n+\ez_1+2N(1/p_\Phi^--1)})^{1/2}\r\|_{L^2(U_j(B_0))}
\frac1{t^{2(\wz M-k)+1}}\,dt\ev\mj_1+\mj_2.
\end{eqnarray*}
For all $\wz\ez>0$, let
$C_1\ev2^{\frac {j_0}2(n+\ez_1+2N(1/p_\Phi^--1))}\|b\|_{L^2(\cx)}$ and
$R_1\ev(\frac {C_1}{\wz\ez})^{\frac1{2(\wz M-k)}}$, then for all
$R>R_1$, we obtain
\begin{eqnarray*}
\mj_1
\ls2^{\frac {j_0}2(n+\ez_1+2N(1/p_\Phi^--1))}
\int_R^\fz\frac{dt}{t^{2(\wz M-k)+1}}\|b\|_{L^2(\cx)}\ls\wz\ez.
\end{eqnarray*}
Letting $C_2\ev r_B^{\mz +2\wz M}$ and
$R_1\ev(\frac {C_2}{\wz\ez})^{\frac1{2(\wz M-k)}}$, then for all
$R>R_1$, we know that
\begin{eqnarray*}
\mj_2&&
\ls\sum_{j=j_0+1}^\fz 2^{\frac j2(n+\ez_1+2N(1/p_\Phi^--1))}\\
&&\hs\times\sum_{i=-d}^d\Bigg\{
\int_R^\fz\lf\|(t^2L^*)^{M+\wz M-k+1}e^{-2t^2L^*}
(\chi_{\wz U_{j+i}(B)} b)
\r\|_{L^2(U_{j+i}(B))}
\frac1{t^{2(\wz M-k)+1}}\,dt\\
&&\hs+\int_R^\fz\lf\|(t^2L^*)^{M+\wz M-k+1}e^{-2t^2L^*}
(\chi_{(\wz U_{j+i}(B))^\com} b)
\r\|_{L^2(U_{j+i}(B))}
\frac1{t^{2(\wz M-k)+1}}\,dt\Bigg\},
\end{eqnarray*}
where $\wz U_{j+i}(B)\ev 2^{j+i+1}B\bh2^{j+i-1}B$. Then, since
\begin{eqnarray*}
&&\int_R^\fz\lf\|(t^2L^*)^{M+\wz M-k+1}e^{-2t^2L^*}
(\chi_{\wz U_{j+i}(B)} b)
\r\|_{L^2(U_{j+i}(B))}
\frac1{t^{2(\wz M-k)+1}}\,dt\\
&&\hs\ls\frac1{R^{2(\wz M-k)}}\|b\|_{L^2(\wz U_{j+i}(B))}
\ls2^{-\frac{j}{2}(n+\ez_1+2N(1/p_\Phi^--1))}\wz \ez,
\end{eqnarray*}
and $\int_R^\fz\lf\|(t^2L^*)^{M+\wz M-k+1}e^{-2t^2L^*}
(\chi_{(\wz U_{j+i}(B))^\com} b)
\r\|_{L^2(U_{j+i}(B))}
\frac1{t^{2(\wz M-k)+1}}\,dt$
satisfies the same estimate, we see that $\mj_2\ls\wz\ez$.
Thus, $\lim_{R\to\fz}\mj=0$.

To consider $\mh$, let $\wz f\ev[I-(I+L)^{-1}]^Mf$. Then
\begin{eqnarray*}
S_{M+1}&&\ev\lf\la \wz f,
\int_0^\dz (t^2L^*)^{M+1}e^{-2t^2L^*}(L^*)^{\wz M-k}b\dt\r\ra\\
&&=-\frac14\lf\la \wz f,
\int_0^\dz (t^2L^*)^{M}\frac{\pa}{\pa t}( e^{-2t^2L^*})(L^*)^{\wz M-k}b\dt\r\ra\\
&&=-\frac14\lf\la \wz f,
(\dz^2L^*)^{M} e^{-2\dz^2L^*}(L^*)^{\wz M-k}b\r\ra
+\frac M2\lf\la \wz f,
\int_0^\dz (t^2L^*)^{M} e^{-2t^2L^*}(L^*)^{\wz M-k}b\dt\r\ra.
\end{eqnarray*}
Thus,
\begin{eqnarray*}
S_{M+1}&&=
-\frac14\lf\la \wz f,(\dz^2L^*)^{M} e^{-2\dz^2L^*}(L^*)^{\wz M-k}b\r\ra
+\frac M2S_M\\
&&=\sum_{\ell=1}^M\frac{-M!}{2^{\ell+1}(M-\ell+1)!}
\lf\la \wz f,(\dz^2L^*)^{M-\ell+1} e^{-2\dz^2L^*}(L^*)^{\wz M-k}b\r\ra
+\frac{M!}{2^M}S_1.
\end{eqnarray*}
For all $\ell\in\{1,\cdots,M\}$, from H\"older's inequality, we infer that
\begin{eqnarray*}
&&\lf|\lf\la \wz f,(\dz^2L^*)^{M-\ell+1} e^{-2\dz^2L^*}(L^*)^{\wz M-k}b\r\ra\r|\\
&&\hs\ls
\lf\{\int_\cx\frac{|(I-(I+L)^{-1})^Mf(x)|^2}{1+[d(x,x_0)]^{n+\ez_1+2N(1/p_\Phi^--1)}}
\,d\mu(x)\r\}^{1/2}\\
&&\hs\hs\times
\lf\{\int_\cx\lf|
(\dz^2L^*)^{M-\ell+1}e^{-2\dz^2L^*}(L^*)^{\wz M-k}b(x)\r|^2
\lf(1+[d(x,x_0)]^{n+\ez_1+2N(1/p_\Phi^--1)}\r)\,d\mu(x)\r\}^{1/2}\\
&&\hs\ls
2^{\frac{j_0}2[n+\ez_1+2N(1/p_\Phi^--1)]}
\lf\|(\dz^2L^*)^{M-\ell+1}e^{-2\dz^2L^*}(L^*)^{\wz M-k}b\r\|_{L^2(2^{j_0}B_0)}\\
&&\hs\hs+\sum_{j=j_0+1}^\fz
2^{\frac{j}2[n+\ez_1+2N(1/p_\Phi^--1)]}\\
&&\hs\hs\times\Bigg\{
\lf\|(\dz^2L^*)^{M-\ell+1}e^{-2\dz^2L^*}
\lf(\chi_{\bigcup_{i=j-d-1}^{j+d+1}U_i(B)}
(L^*)^{\wz M-k}b\r)\r\|_{L^2(U_j(B_0))}\\
&&\hs\hs+\lf\|(\dz^2L^*)^{M-\ell+1}e^{-2\dz^2L^*}
\lf(\chi_{(\bigcup_{i=j-d-1}^{j+d+1}U_i(B))^\com}
(L^*)^{\wz M-k}b\r)\r\|_{L^2(U_j(B_0))}\Bigg\}.
\end{eqnarray*}
By the $L^2$-functional calculus, we see that
$\lim_{\dz\to0}(\dz^2L^*)^{M-\ell+1}e^{-2\dz^2L^*}(L^*)^{\wz M-k}b=0$ in
$L^2(\cx)$ and, by Lemma \ref{l2.2}, we know that
\begin{eqnarray*}
&&\sum_{j=j_0+1}^\fz
2^{\frac{j}2[n+\ez_1+2N(1/p_\Phi^--1)]}\Bigg\{
\lf\|(\dz^2L^*)^{M-\ell+1}e^{-2\dz^2L^*}
\lf(\chi_{\bigcup_{i=j-d-1}^{j+d+1}U_i(B)}(L^*)^{\wz M-k}b\r)\r\|_{L^2(U_j(B_0))}\\
&&\hs\hs+\lf\|(\dz^2L^*)^{M-\ell+1}e^{-2\dz^2L^*}
\lf(\chi_{(\bigcup_{i=j-d-1}^{j+d+1}
U_i(B))^\com}(L^*)^{\wz M-k}b\r)\r\|_{L^2(U_j(B_0))}\Bigg\}\\
&&\hs\ls\sum_{j=j_0+1}^\fz
2^{\frac{j}2[n+\ez_1+2N(1/p_\Phi^--1)]}
\lf[\|(L^*)^{\wz M-k}b\|_{L^2(\bigcup_{i=j-d-1}^{j+d+1}U_i(B))}
+e^{-\frac{2^jr_B}{\dz}}\|(L^*)^{\wz M-k}b\|_{L^2(\cx)}\r]\\
&&\hs\ls\wz\ez.
\end{eqnarray*}
From
\begin{eqnarray*}
S_1=\lf\la \wz f,
\int_0^\dz (t^2L^*)e^{-2t^2L^*}(L^*)^{\wz M-k}b\dt\r\ra
=\lf\la \wz f, \lf(e^{-2\dz^2L^*}-I\r)(L^*)^{\wz M-k}b\r\ra,
\end{eqnarray*}
and
$$\lim_{\dz\to0}\lf\|\lf(e^{-2\dz^2L^*}-I\r)(L^*)^{\wz M-k}b\r\|_{L^2(\cx)}=0,$$
it follows that $\lim_{\dz\to0}H=0$, which completes the proof of
Proposition \ref{p3.3}.
\end{proof}

Instead of \cite[Proposition 4.6]{jy} by Proposition \ref{3.3} here,
repeating the proof of \cite[Corollary 4.3]{jy}, we obtain the following
Lemma \ref{3.2}. The details are omitted.

\begin{lemma}\label{l3.2}
Let $L$, $\Phi$, $\rho$ and $M$ be as in Definition \ref{d3.2} and $\ez\in(0,\fz)$.
If $f\in\bmo$, then for any $\pmx$-molecule $\az$,
there holds
$$\la f, \az\ra=C(M)\iint_\xt\tml f(x) \ov{\tlx\az(x)}\dxt.$$
\end{lemma}

Recall that a measure $d\mu$ on $\xt$ is called a \emph{$\rho$-Carleson
measure} if
\begin{equation*}
\|d\mu\|_\rho\ev
\sup_{B\st \cx}\lf\{\frac1{\mu(B)[\rho(\mu(B))]^2}
\iint_{\wh B}|d\mu|\r\}^{1/2}<\fz,
\end{equation*}
where the supremum is taken over all balls $B$ of $\cx$.

Using Theorem \ref{t3.2} and Proposition \ref{p3.2}, similar to
the proof of \cite[Theorem 4.2]{jy}, we obtain
the following $\rho$-Carleson measure characterization of $\bmo$.

\begin{theorem}\label{t3.3}
Let $L$, $\Phi$, $\rho$ and $M$ be as in Definition \ref{d3.2}.
Fix $x_0\in\cx$. Then the following are equivalent:

{\rm (i)} $f\in\bmo$;

{\rm (ii)} $f\in\cm_{\Phi,L}^{M}(\cx)$ satisfies that
$$\int_\cx
\frac{|(I-(I+L)^{-1})^Mf(x)|^2}{1+[d(x,x_0)]^{n+\ez_1+2N(1/p_\Phi^--1)}}
\,d\mu(x)<\fz$$
for some $\ez_1\in(0,\fz)$, and $d\mu_f$ is a $\rho$-Carleson measure,
where $d\mu_f$ is defined by $d\mu_f(x,t)\ev|(t^2L)^{M}e^{-t^2L}f(x)|^2\dxt$
for all $(x,t)\in\xt$.

Moreover, $\|d\mu_f\|_\rho$ is equivalent to $\|f\|_\bmo$.
\end{theorem}

\begin{proof}
It follows from Proposition \ref{p3.1} and the proof of Lemma \ref{l3.2}
that (i) implies (ii).

To show that (ii) implies (i),
let $\wz M>M+\ez_1+\frac n4+\frac N2(\frac1{p_\Phi^-}-1)$.
From Proposition \ref{p3.3}, we deduce that
$$\la f,g\ra=C(M)\iint_\xt\tml f(x) \ov{\tlx g(x)}\dxt,$$
where $g$ is any finite combination of $\pmx$-molecules.
Then $t^2L^*e^{-t^2L^*}g\in\tx$. By Lemma \ref{l3.1}, there exist
$\{\lz_j\}_{j=1}^\fz\st\cc$ and $\tx$-atoms $\{a_j\}_{j=1}^\fz$
supported in $\{\wh B_j\}_{j=1}^\fz$ such that \eqref{3.1} and
\eqref{3.2} hold. This, together with Fatou's lemma and H\"older's
inequality, implies that
\begin{eqnarray*}
|\la f,g\ra|&&=\lf|C(M)\iint_\xt\tml f(x) \ov{\tlx g(x)}\dxt\r|\\
&&\ls\sum_j|\lz_j|\int_0^\fz\int_\cx|\tml f(x)\ov{a_j(x,t)}|\dxt\\
&&\ls\sum_j|\lz_j|\|a_j\|_{T_2^2(\cx)}
\lf(\iint_{\wh B_j}|\tml f(x)|^2\dxt\r)^{1/2}\\
&&\ls\sum_j|\lz_j|\|d\mu_f\|_\rho\ls\|\tmlx g\|_\tx\|d\mu_f\|_\rho
\sim\|g\|_\hxx\|d\mu_f\|_\rho.
\end{eqnarray*}
By this and Theorem \ref{t3.2}, we conclude that $f\in(\hxx)^*=\bmo$,
which completes the proof of Theorem \ref{t3.3}.
\end{proof}

Now we introduce the space $\vmo$.

\begin{definition}\rm\label{d3.3}
Let $L$, $\Phi$, $\rho$ and $M$ be as in Definition \ref{d3.2}. An
element $f\in\bmo$ is said to be in the \emph{space} $\vmol$ if it satisfies the
following limiting conditions $\gz_1(f) =\gz_2(f)=\gz_3(f)=0$, where
$x_0\in\cx$ is a fixed point, $c\in(0,\fz)$,
$$\gz_1(f)\equiv
\lim_{c\to0}\sup_{B:\,r_B\le c}\frac1{\rho(\mu(B))}\lf[\frac1{\mu(B)}
\int_B|(I-e^{-r_B^2L})^Mf(x)|^2\,d\mu(x)\r]^{1/2},$$
$$\gz_2(f)\equiv
\lim_{c\to\fz}\sup_{B:\,r_B\ge c}\frac1{\rho(\mu(B))}\lf[\frac1{\mu(B)}
\int_B|(I-e^{-r_B^2L})^Mf(x)|^2\,d\mu(x)\r]^{1/2},$$
and
$$\gz_3(f)\equiv\lim_{c\to\fz}\sup_{B:\,B\subset[B(x_0,c)]^
\complement}\frac1{\rho(\mu(B))}\lf[\frac1{\mu(B)}
\int_B|(I-e^{-r_B^2L})^Mf(x)|^2\,d\mu(x)\r]^{1/2}.$$
For any $f\in\vmol$, define $\|f\|_{\vmol}\equiv\|f\|_{\bmo}$.
\end{definition}

\begin{definition}\rm\label{d3.4}
Let $\Phi$ satisfy Assumption $(\Phi)$ and
$\rho$ be as in \eqref{2.11}. The \emph{space}
$\txv$ is defined to be the space of all $f\in\txz$
satisfying $\eta_1(f)=\eta_2(f)=\eta_3(f)=0$ with the same norm as
the space $\txz$, where $x_0\in\cx$ is a fixed point, $c\in(0,\fz)$,
$$\eta_1(f)\equiv
\lim_{c\to0}\sup_{B:\,r_B\le c}\frac1{\rho(\mu(B))}\lf[\frac1{\mu(B)}
\iint_{\wh B}|f(y,t)|^2\,\frac{d\mu(y)\,dt}t\r]^{1/2},$$
$$\eta_2(f)\equiv
\lim_{c\to\fz}\sup_{B:\,r_B\ge c}\frac1{\rho(\mu(B))}\lf[\frac1{\mu(B)}
\iint_{\wh B}|f(y,t)|^2\,\frac{d\mu(y)\,dt}t\r]^{1/2},$$
and
$$\eta_3(f)\equiv\lim_{c\to\fz}\sup_{B:\,B\subset[B(x_0,c)]^
\complement}\frac1{\rho(\mu(B))}\lf[\frac1{\mu(B)}
\iint_{\wh B}|f(y,t)|^2\,\frac{d\mu(y)\,dt}t\r]^{1/2}.$$
\end{definition}
It is easy to see that $\txv$ is a \emph{closed linear subspace} of $\txz$.

Further, denote by $\txy$ the {\it space of all $f\in\txz$ with
$\eta_1(f)=0$}, and $\txb$ the {\it space of all $f\in T_2^2(\cx)$
with bounded support}. Obviously, we have $\txb\st\txv\st\txy$.
Finally, denote by $\txl$ the {\it closure of $\txb$ in the space
$\txy$}.

\begin{lemma}\label{l3.3}
Let $L$ and $\Phi$ be as in Definition \ref{d3.1}, and $\txv$ and
$\txl$ defined as above. Then $\txv$ and $\txl$ coincide with
equivalent norms.
\end{lemma}

\begin{proof}
Since $\txb\st\txv$ and $\txv$ is a closed linear subspace
of $\txz$, we conclude that $\txl= {\txb}\st\txv$.

Conversely, for any $f\in\txv$, by the definition of $\txv$, for any $\ez>0$,
there exist positive constants $a_0$, $b_0$ and $c_0$ such that
\begin{equation}\label{3.7}
\sup_{B:\,r_B\le a_0}\frac1{\mu(B)[\rho(\mu(B))]^2}
\iint_{\wh B}|f(y,t)|^2\,\frac{d\mu(y)\,dt}t<\ez,
\end{equation}
\begin{equation}\label{3.8}
\sup_{B:\,r_B\ge b_0}\frac1{\mu(B)[\rho(\mu(B))]^2}
\iint_{\wh B}|f(y,t)|^2\,\frac{d\mu(y)\,dt}t<\ez,
\end{equation}
and
\begin{equation}\label{3.9}
\sup_{B:\,B\subset[B(x_0,c_0)]^
\complement}\frac1{\mu(B)[\rho(\mu(B))]^2}
\iint_{\wh B}|f(y,t)|^2\,\frac{d\mu(y)\,dt}t<\ez.
\end{equation}

Let $K_0\ev\max\{a_0^{-1},b_0,c_0\}$ and, for all $(y,t)\in\xt$,
\begin{equation*}
g(y,t)\ev f(y,t)\chi_{B(x_0,2K_0)\times((2K_0)^{-1},2K_0)}(y,t).
\end{equation*}
Obviously, $g\in\txb$. To complete the proof of Lemma \ref{l3.3},
we need show that
$$\|f-g\|_\txz^2\ls \ez.$$
We consider the following three cases
for all balls $B$ in \eqref{3.7}, \eqref{3.8} and \eqref{3.9}.

Case (i) $r_B<a_0$ or $r_B>b_0$. In this case,
from \eqref{3.7} and \eqref{3.8}, we deduce that
\begin{equation*}
\|f-g\|_\txz^2\le\frac2{\mu(B)[\rho(\mu(B))]^2}
\iint_{\wh B}|f(y,t)|^2\,\frac{d\mu(y)\,dt}t\le2\ez.
\end{equation*}

Case (ii) $a_0\le r_B\le b_0$ and $B\st [B(x_0,c_0)]^\com$.
In this case, by \eqref{3.9}, we conclude that
\begin{equation*}
\|f-g\|_\txz^2\le\frac2{\mu(B)[\rho(\mu(B))]^2}
\iint_{\wh B}|f(y,t)|^2\,\frac{d\mu(y)\,dt}t\le2\ez.
\end{equation*}

Case (iii) $a_0\le r_B\le b_0$ and $B\cap B(x_0,c_0)\neq\emptyset$.
In this case, we have
\begin{eqnarray*}
\iint_{\wh B}|f(y,t)-g(y,t)|^2\,\frac{d\mu(y)\,dt}t
&&\le
\int_0^{(2K_0)^{-1}}\int_B|f(y,t)|^2\dyt\\
&&\le
\int_0^{(2K_0)^{-1}}\int_{B(x_B,2^ka_0)}|f(y,t)|^2\dyt,
\end{eqnarray*}
where $x_B$ is the \emph{center} of $B$ and $k$
the \emph{smallest integer} such that $2^ka_0>r_B$.
Then, by Lemma \ref{l2.1}, we pick a family of balls with
the same radius $a_0$, $\{B(x_{B,i},a_0)\}_{i=1}^{N_k}$, such that
$B(x_B,2^ka_0)\st\cup_{i=1}^{N_k}B(x_{B,i},a_0)$, ${N_k}\ls 2^{kn}$ and
$\sum_{i=1}^{N_k}\chi_{B(x_{B,i},a_0)}\ls1$. Therefore, combining the fact that
$\rho$ is an increasing function, we obtain
\begin{eqnarray*}
\iint_{\wh B}|f(y,t)-g(y,t)|^2\,\frac{d\mu(y)\,dt}t
&&\le
\int_0^{(2K_0)^{-1}}\int_{\cup_{i=1}^{N_k}B(x_{B,i},a_0)}|f(y,t)|^2\dyt\\
&&\le
\sum_{i=1}^{N_k}\iint_{\wh B(x_{B,i},a_0)}|f(y,t)|^2\dyt\\
&&\ls
\ez\sum_{i=1}^{N_k}\mu(B(x_{B,i},a_0))[\rho(\mu(B(x_{B,i},a_0)))]^2\\
&&\ls
\ez[\rho(\mu(B))]^2\sum_{i=1}^{N_k}\mu(B(x_{B,i},a_0))
\ls\ez\mu(B)[\rho(\mu(B))]^2,
\end{eqnarray*}
which completes the proof of Lemma \ref{l3.3}.
\end{proof}

\begin{definition}\rm\label{d3.5}
Let $L$, $\Phi$, $\rho$ and $M$ be as in Definition \ref{d3.2}.
The \emph{space} $\tvmol$ is defined to be the space of all elements
$f\in\bmol$ that satisfy the following limiting conditions
$\wz\gz_1(f)=\wz\gz_2(f)=\wz\gz_3(f)=0$, where $c\in(0,\fz)$,
$$\wz\gz_1(f)\equiv
\lim_{c\to0}\sup_{B:\,r_B\le c}\frac1{\rho(\mu(B))}\lf[\frac1{\mu(B)}
\int_B|(I-[I+r_B^2L]^{-1})^Mf(x)|^2\,d\mu(x)\r]^{1/2},$$
$$\wz\gz_1(f)\equiv
\lim_{c\to\fz}\sup_{B:\,r_B\ge c}\frac1{\rho(\mu(B))}\lf[\frac1{\mu(B)}
\int_B|(I-[I+r_B^2L]^{-1})^Mf(x)|^2\,d\mu(x)\r]^{1/2},$$
and
$$\wz\gz_1(f)\equiv\lim_{c\to\fz}\sup_{B:\,B\subset[B(0,c)
]^\complement}\frac1{\rho(\mu(B))}\lf[\frac1{\mu(B)}
\int_B|(I-[I+r_B^2L]^{-1})^Mf(x)|^2\,d\mu(x)\r]^{1/2}.$$
\end{definition}

\begin{proposition}\label{p3.4}
Let $L$, $\Phi$, $\rho$ and $M$ be as in Definition \ref{d3.2}. Then
$f\in\vmol$ if and only if $f\in\tvmol$.
\end{proposition}

\begin{proof}
Suppose that $f\in\tvmol$. To see $f\in\vmol$, it suffices to
show that
\begin{equation}\label{3.10}
\frac1{\rho(\mu(B))[\mu(B)]^{1/2}}\lf[\int_B\lf|(I-e^{-r_B^2L})^Mf(x)\r|^2
\,d\mu(x)\r]^{1/2}\ls\sum_{k=0}^\fz2^{-k}\dz_k(f,B),
\end{equation}
where
\begin{eqnarray}\label{3.11}
\dz_k(f,B)\ev&&\sup_{\{B'\st2^{k+1}B:\,r_{B'}\in[2^{-1}r_B,r_B]\}}
\frac1{\rho(\mu(B))[\mu(B)]^{1/2}}\\
&&\times\lf[\int_B\lf|(I-[I+r_B^2L]^{-1})^Mf(x)\r|^2\,d\mu(x)\r]^{1/2}.\noz
\end{eqnarray}

Indeed, since $f\in\tvmol$, by Definition \ref{d3.5} and Proposition
\ref{p3.1}, we conclude that $\dfb\ls\|f\|_\bmo$ and for all $k\in\zz_+$,
$$\lim_{c\to0}\sup_{B:\,r_B\le c}\dfb=\lim_{c\to\fz}\sup_{B:\,r_B\ge c}
\dfb=\lim_{c\to\fz}\sup_{B:\,B\subset[B(x_0,c)]^\complement}\dfb=0.$$
Then by the dominated convergence theorem for series, we have
\begin{eqnarray*}
\gz_1(f)&&=\lim_{c\to0}\sup_{B:\,r_B\le c}\frac1{\rho(\mu(B))[\mu(B)]^{1/2}}
\lf[\int_B\lf|\lf(I-e^{-r_B^2L}\r)^Mf(x)\r|^2\,d\mu(x)\r]^{1/2}\\
&&\ls\sum_{k=1}^\fz2^{-k}\lim_{c\to0}\sup_{B:\,r_B\le c}\dfb=0.
\end{eqnarray*}
Similarly we see that $\gz_2(f)=\gz_3(f)=0$, and hence $f\in\vmol$.

Let us now prove \eqref{3.10}. Write
\begin{equation}\label{3.12}
f=\lf(I-[I+r_B^2L]^{-1}\r)^Mf+\lf\{I-\lf(I-[I+r_B^2L]^{-1}\r)^M\r\}f\ev f_1+f_2.
\end{equation}
By Lemma \ref{l2.2}, we have
\begin{eqnarray}\label{3.13}
&&\lf\|\Iem f_1\r\|_{L^2(B)}\\
&&\hs\le\sum_{k=0}^\fz\lf\|\Iem(f_1\chi_{U_k(B)})\r\|_{L^2(B)}\ls
\sum_{k=0}^\fz e^{-c2^{2k}}\lf\|f_1\chi_{U_k(B)}\r\|_{L^2(\cx)}\noz\\
&&\hs\ls\rho(\mu(B))[\mu(B)]^{1/2}\sum_{k=0}^\fz e^{-c2^{2k}}2^{kn}\dfb\noz\\
&&\hs\ls\rho(\mu(B))[\mu(B)]^{1/2}\sum_{k=0}^\fz 2^{-k}\dfb,\noz
\end{eqnarray}
where $U_k(B)$ for all $k\in\zz_+$ is as in \eqref{2.5},
$c$ is a positive constant and the third inequality follows
from Lemma \ref{l2.1} that there exists a collection, $\{B_{k,1},B_{k,2},
\cdots,B_{k,N_k}\}$, of balls such that each ball $B_{k,i}$ is of
radius $r_B$, $B(x_B,2^{k+1}r_B)\st\cup_{i=1}^{N_k}B_{k,i}$ and
$N_k\ls2^{nk}$.

To estimate the remaining term, by the formula that
\begin{equation}\label{3.14}
I-\Ilm=\sum_{j=1}^M\frac{M!}{j!(M-j)!}(r_B^2L)^{-j}\Ilm
\end{equation}
(which relies on the fact that $(I-(I+r^2L)^{-1})(r^2L)^{-1}=
(I+r^2L)^{-1}$ for all $r\in(0,\fz)$), and Minkowski's
inequality, we obtain
\begin{eqnarray}\label{3.15}
&&\lf\|\Iem f_2\r\|_{L^2(B)}\\
&&\hs\ls\sum_{j=1}^M\lf\{\int_B\lf|\lf(I-e^{-r_B^2L}\r)^{M-j}
\lf[-\int_0^{r_B}\frac s{r_B^2}e^{-s^2L}\,ds\r]^jf_1(x)\r|^2
\,d\mu(x)\r\}^{1/2}\noz\\
&&\hs\ls\sum_{j=1}^M\sum_{i=0}^{M-j}\int_0^{r_B}\cdots\int_0^{r_B}
\frac{s_1}{r_B^2}\cdots\frac{s_j}{r_B^2}\|e^{-(ir_B^2+s_1^2+
\cdots+s_j^2)L}f_1\|_{L^2(B)}\,ds_1\cdots ds_j\noz\\
&&\hs\ls\sum_{j=1}^M\sum_{i=0}^{M-j}\int_0^{r_B}\!\!\cdots\!\!\int_0^{r_B}
\frac{s_1}{r_B^2}\!\cdots\!\frac{s_j}{r_B^2}\sum_{k=0}^\fz e^{-\frac{c
(2^kr_B)^2}{ir_B^2+s_1^2+\cdots+s_j^2}}\|f_1\chi_{U_k(B)}\|_{L^2(\cx)}
\,ds_1\cdots ds_j\noz\\
&&\hs\ls\rho(\mu(B))[\mu(B)]^{1/2}\sum_{k=0}^\fz
e^{-\frac{c2^{2k}}M}2^{kn}\dfb\noz\\
&&\hs\ls\rho(\mu(B))[\mu(B)]^{1/2}\sum_{k=0}^\fz 2^{-k}\dfb,\noz
\end{eqnarray}
where $c$ is a positive constant and in the penultimate inequality,
we used the fact that $\int_0^{r_B}\cdots\int_0^{r_B}
\frac{s_1}{r_B^2}\cdots\frac{s_j}{r_B^2}\,ds_1\cdots ds_j\sim1$.
Combining the estimates \eqref{3.13} and \eqref{3.15}, we obtain
\eqref{3.10}, which further implies that $\tvmol\st\vmol$.

By borrowing some ideas from the proof of \cite[Lemma 8.1]{hm09},
then similar to the proof above, we conclude that $\vmol\st\tvmol$
and the details are omitted. This finishes the proof of Proposition
\ref{p3.4}.
\end{proof}

We now characterize the space $\vmol$ via the tent space.

\begin{theorem}\label{t3.4}
Let $L$, $\Phi$ and $\rho$ be as in Definition \ref{d3.1}, $M$, $M_1\in\nn$
and $M_1\ge M>\mz$. Then the following are equivalent:

{\rm (i)} $f\in\vmol$;

{\rm (ii)} $f\in\cm_{\Phi,L}^{M_1}(\cx)$ and $(t^2L)^{M_1}e^{-t^2L}f\in\txv$.

Moreover, $\|(t^2L)^{M_1}e^{-t^2L}f\|_\txz$ is equivalent to $\|f\|_\bmo$.
\end{theorem}

\begin{proof}
We first show that (i) implies (ii).
Let $f\in\vmol$. By Proposition \ref{p3.2}, we know that $\tmy f\in\txz$.
To see that $\tmy f\in\txv$, we claim that it suffices to
show that for all balls $B$,
\begin{equation}\label{3.16}
\frac1{\rho(\mu(B))[\mu(B)]^{1/2}}\lf[\iint_{\wh B}|\tmy f(x)|^2
\,\frac{d\mu(x)\,dt}t\r]^{1/2}\ls\sum_{k=0}^\fz2^{-k}\dz_k(f,B),
\end{equation}
where $\dfb$ is as in \eqref{3.11}.
Indeed, since $f\in\vmol=\tvmol$, we conclude that for each $k\in\nn$,
$\dfb\ls\|f\|_\bmo$ and
$$\lim_{c\to0}\sup_{B:\,r_B\le c}\dfb=\lim_{c\to\fz}\sup_{B:\,r_B\ge c}
\dfb=\lim_{c\to\fz}\sup_{B:\,B\subset[B(x_0,c)]^\complement}\dfb=0.$$
Then from the dominated convergence theorem for series, we infer that
\begin{eqnarray*}
\eta_1(f)&&=\lim_{c\to0}\sup_{B:\,r_B\le c}\frac1{\rho(\mu(B))[\mu(B)]^{1/2}}
\lf[\iint_{\wh B}\lf|\tmy f(x)\r|^2\,\frac{d\mu(x)\,dt}t\r]^{1/2}\\
&&\ls\sum_{k=1}^\fz2^{-k}\lim_{c\to0}\sup_{B:\,r_B\le c}\dfb=0.
\end{eqnarray*}
Similarly we see that $\eta_2(f)=\eta_3(f)=0$, and hence $\tmy f\in\txv$.

Let us now prove \eqref{3.16}. Write $f\ev f_1+f_2$ as in \eqref{3.12}.
Then by Lemmas \ref{l2.2} and \ref{l2.3},
similar to the estimate of \eqref{3.13}, we have
\begin{eqnarray}\label{3.17}
&&\lf\{\iint_{\wh B}\lf|\tmy f_1(x)\r|^2\dxt\r\}^{1/2}\\
&&\hs\le
\sum_{k=0}^\fz\lf\{\iint_{\wh B}\lf|\tmy (f_1\cub)(x)\r|^2\dxt\r\}^{1/2}\noz\\
&&\hs\ls\|f_1\|_{L^2(4B)}+\sum_{k=3}^\fz\lf[\int_0^{r_B}\exp\lf\{
-\frac{(2^kr_B)^2}{ct^2}\r\}\,\frac{dt}t\r]^{1/2}\lf\|f_1\cub\r\|_{L^2(\cx)}\noz\\
&&\hs\ls\|f_1\|_{L^2(4B)}+\sum_{k=3}^\fz\lf\{\int_0^{r_B}\lf[
\frac{t^2}{(2^kr_B)^2}\r]^{n+1}
\,\frac{dt}t\r\}^{1/2}\lf\|f_1\cub\r\|_{L^2(\cx)}\noz\\
&&\hs\ls\rho(\mu(B))[\mu(B)]^{1/2}\sum_{k=0}^\fz 2^{-k}\dfb,\noz
\end{eqnarray}
where $U_k(B)$ for all $k\in\zz_+$ is as in \eqref{2.5} and
$c$ is a positive constant. Applying \eqref{3.14}, Lemma \ref{l2.2}
and $M_1>M$ to $f_2$, we see that
\begin{eqnarray}\label{3.18}
&&\lf\{\iint_{\wh B}\lf|\tmy f_2(x)\r|^2\dxt\r\}^{1/2}\\
&&\hs\ls
\sum_{j=1}^M\lf\{\iint_{\wh B}\lf|\tmy (r_B^2L)^{-j}f_1(x)\r|^2\dxt\r\}^{1/2}\noz\\
&&\hs\ls\sum_{j=1}^M\sum_{k=0}^\fz\lf\{\iint_{\wh B}\lf[\frac{t^2}
{r_B^2}\r]^{2j}\lf|(t^2L)^{M_1-j}e^{-t^2L} (f_1\cub)(x)\r|^2\dxt\r\}^{1/2}\noz\\
&&\hs\ls\sum_{j=1}^M\lf\{\sum_{k=0}^2\lf[\int_0^{r_B}\lf(\frac{t^2}
{r_B^2}\r)^{2j}\dt\r]^{1/2}\|f_1\|_{L^2(4B)}\r.\noz\\
&&\hs\hs\lf.+\sum_{k=3}^\fz\lf[\int_0^{r_B}\exp\lf\{
-\frac{(2^kr_B)^2}{ct^2}\r\}
\,\frac{dt}t\r]^{1/2}\lf\|f_1\cub\r\|_{L^2(\cx)}\r\}\noz\\
&&\hs\ls\|f_1\|_{L^2(4B)}+\sum_{k=3}^\fz\lf\{\int_0^{r_B}\lf[
\frac{t^2}{(2^kr_B)^2}\r]^{n+1}
\,\frac{dt}t\r\}^{1/2}\lf\|f_1\cub\r\|_{L^2(\cx)}\noz\\
&&\hs\ls\rho(\mu(B))[\mu(B)]^{1/2}\sum_{k=0}^\fz 2^{-k}\dfb.\noz
\end{eqnarray}

The estimates \eqref{3.17} and \eqref{3.18} imply \eqref{3.16},
which completes the proof that (i) implies (ii).

Conversely, let $f\in\cm_{\Phi,L}^{M_1}(\cx)$ and $\tmy f\in\txv$.
By Proposition \ref{p3.2}, we conclude that $f\in\bmo$. For any ball $B$,
write
\begin{eqnarray*}
\lf(\int_B\lf|\Iem f(x)\r|^2\,d\mu(x)\r)^{1/2}&&=\sup_{\|g\|_{L^2(B)}\le1}
\lf|\int_B\Iem f(x)\ov{g(x)}\,d\mu(x)\r|\\
&&=\sup_{\|g\|_{L^2(B)}\le1}\lf|\int_B f(x)\ov{\Iemx g(x)}\,d\mu(x)\r|.
\end{eqnarray*}
Notice that for any $g\in L^2(B)$, $\iemx g$ is a multiple of a
$(\Phi,\,M,\,\ez)_{L^*}$-molecule; see \cite[p. 43]{hm09}.
Then by Lemma \ref{l3.2}
and H\"older's inequality, we obtain
\begin{eqnarray*}
&&\lf[\int_B\lf|\Iem f(x)\r|^2\,d\mu(x)\r]^{1/2}\\
&&\hs\sim\sup_{\|g\|_{L^2(B)}\le1}
\lf|\iint_\xt\tmy f(x)t^2L^*e^{-t^2L^*}\ov{\Iemx g(x)}\dxt\r|\\
&&\hs\sim\sum_{k=0}^\fz\lf\{\iint_{V_k(B)}\lf|\tmy f(x)\r|^2\dxt\r\}^{1/2}\\
&&\hs\hs\times\sup_{\|g\|_{L^2(B)}\le1}\lf\{\iint_{V_k(B)}\lf|t^2L^*e^{-t^2L^*}
\Iemx g(x)\r|^2\dxt\r\}^{1/2}\\
&&\hs\ev\sum_{k=0}^\fz\sigma_k(f,B)\mi_k,
\end{eqnarray*}
where $V_0(B)\ev\wh B$ and $V_k(B)\ev(\wh{2^kB})\bh(\wh{2^{k-1}B})$ for
$k\in\nn$. In what follows, for $k\ge2$, let $V_{k,1}\ev(\wh{2^kB})\bh
(2^{k-2}B\times(0,\fz))$ and $V_{k,2}\ev V_k(B)\bh V_{k,1}(B)$.

For $k\in\{0,1,2\}$, by Lemmas \ref{l2.2} and \ref{l2.3}, we conclude that
\begin{eqnarray*}
\mi_k&&=\sup_{\|g\|_{L^2(B)}\le1}\lf\{\iint_{V_k(B)}\lf|t^2L^*e^{-t^2L^*}
\Iemx g(x)\r|^2\dxt\r\}^{1/2}\\
&&\ls\sup_{\|g\|_{L^2(B)}\le1}\lf\|\Iemx g\r\|_{L^2(\cx)}\ls1.
\end{eqnarray*}
Now for $k\ge3$, write
\begin{eqnarray*}
\mi_k&&\ls\sup_{\|g\|_{L^2(B)}\le1}\lf\{\iint_{V_{k,1}(B)}\lf|t^2L^*e^{-t^2L^*}
\Iemx g(x)\r|^2\dxt\r\}^{1/2}\\
&&\hs+
\sup_{\|g\|_{L^2(B)}\le1}\lf\{\iint_{V_{k,2}(B)}\cdots\r\}^{1/2}
\ev\mi_{k,1}+\mi_{k,2}.
\end{eqnarray*}
Since for any $(y,t)\in\ve$, $t\ge2^{k-2}r_B$, from Minkowski's inequality
and Lemmas \ref{l2.2} and \ref{l2.3}, it follows that
\begin{eqnarray*}
\mi_{k,2}&&=\sup_{\|g\|_{L^2(B)}\le1}\lf\{\iint_{V_{k,2}(B)}\lf|t^2L^*e^{-t^2L^*}
\Iemx g(x)\r|^2\dxt\r\}^{1/2}\\
&&=\sup_{\|g\|_{L^2(B)}\le1}\lf\{\iint_{V_{k,2}(B)}\lf|t^2L^*e^{-t^2L^*}
\lf[-\int_0^{r_B^2}L^*e^{-sL^*}\,ds\r]^M
g(x)\r|^2\dxt\r\}^{1/2}\\
&&\ls\sup_{\|g\|_{L^2(B)}\le1}\int_0^{r_B^2}\cdots\int_0^{r_B^2}\lf\{
\iint_\ve\lf|t^2(L^*)^{M+1}\r.\r.\\
&&\hs\times\lf.\lf.
e^{-(t^2+s_1+\cdots+s_M)L^*}g(x)\r|^2\dxt\r\}^{1/2}ds_1\cdots ds_M\\
&&\ls\sup_{\|g\|_{L^2(B)}\le1}\int_0^{r_B^2}\cdots\int_0^{r_B^2}\lf\{
\int_{2^{k-2}r_B}^{2^{k}r_B}\frac{t^4\|g\|_{L^2(B)}^2}{(t^2+s_1+
\cdots+s_M)^{2(M+1)}}\dt\r\}^{1/2}ds_1\cdots ds_M\\&&\ls2^{-2kM}.
\end{eqnarray*}
Similarly, we see that $\mi_{k,1}\ls2^{-2kM}$.
Let $\wz p_\Phi\in (0, p_\Phi^-)$ such that
$M>\frac n2(\frac1{\wz p_\Phi}-\frac12)$. Combining the above
estimates and the fact that $\rho$ is of upper type $1/\wz p_\Phi-1$,
we finally conclude that
\begin{eqnarray*}
&&\frac1{\rho(\mu(B))[\mu(B)]^{1/2}}\lf[\int_{ B}\lf|\Iem f(x)\r|^2\,d\mu(x)\r]^{1/2}\\
&&\hs\ls\sum_{k=0}^\fz2^{-2kM}\frac1{\rho(\mu(B))[\mu(B)]^{1/2}}\sz_k(f,B)\\
&&\hs\ls\sum_{k=0}^\fz2^{-k[2M-n(1/\wz p_\Phi-1/2)]}\frac{\sz_k(f,B)}
{\rho(\mu(2^kB))[\mu(2^kB)]^{1/2}}.
\end{eqnarray*}
Since $\tmy f\in\txv\st\tx,$
from $M>\frac n2(\frac1{\wz p_\Phi}-\frac12)$
and the dominated convergence theorem for series, we infer that
\begin{eqnarray*}
\gz_1(f)&&=\lim_{c\to0}\sup_{B:\,r_B\le c}\frac1{\rho(\mu(B))[\mu(B)]^{1/2}}
\lf[\int_B\lf|\lf(I-e^{-r_B^2L}\r)^Mf(x)\r|^2\,d\mu(x)\r]^{1/2}\\
&&\ls\sum_{k=1}^\fz2^{-k[2M-n( 1/\wz p_\Phi-1/2)]}\lim_{c\to0}\sup_{B:\,r_B\le c}
\frac{\sz_k(f,B)}{\rho(\mu(2^kB))[\mu(2^kB)]^{1/2}}=0.
\end{eqnarray*}
Similarly, $\gz_2(f)=\gz_3(f)=0$, which implies that $f\in\vmol$,
and hence completes the proof of Theorem \ref{t3.4}.
\end{proof}

\begin{remark}\rm\label{r3.3}
It follows from Theorem \ref{t3.4} that for all $M\in\nn$ and $M>\mz$, the spaces
$\vmol$ coincide with equivalent norms. Thus, in what follows, we
denote the $\vmol$ simply by $\vmo$.
\end{remark}

\section{The Dual Space of $\vmo${\label{s4}}}

\hskip\parindent In this section, we show that the dual space of
$\vmo$ is $\bxx$, where the \emph{space} $\bxx$
denotes the Banach completion of the space
$\hxx$; see Definition \ref{d4.3} and Theorem \ref{t4.2} below.

The proof of the following proposition is similar to that of
\cite[Proposition 4.1]{jyz09}; we omit the details here.

\begin{proposition}\label{p4.1}
Let $\Phi$ satisfy Assumption $(\Phi)$. Then the dual space of $\tx$
is $\txz$. Moreover, the pairing
$$\la f,g\ra\to\int_\xt f(y,t)g(y,t)\dyt$$
for all $f\in\ttx$ and $g\in\txz$
realizes $\txz$ being equivalent to the dual of $\tx$.
\end{proposition}

We now introduce a new tent space $\ttx$ and present some properties.

\begin{definition}\rm\label{d4.1}
Let $p\in(0,1)$. The \emph{space} $\ttx$ is defined to be the space of
all $f=\sum_{j=1}^\fz\lz_j a_j$ in $(\txz)^*$, where $\{a_j\}_{j=1}^\fz$
are $\tx$-atoms and $\{\lz_j\}_{j=1}^\fz\subset\cc$ such
that $\sum_{j=1}^\fz|\lz_j|<\fz$. If $f\in\ttx$,
then define $\|f\|_\ttx\ev\inf\{\sum_{j=1}^\fz|\lz_j|\}$, where
the infimum is taken over all the possible decompositions of $f$
as above.
\end{definition}

By \cite[Lemma 3.1]{hm09},
$\ttx$ is a Banach space. Moreover, from Definition \ref{d4.1},
it is easy to deduce that $\tx$ is dense in $\ttx$; in other
words, $\ttx$ is a \emph{Banach completion} of $\tx$.

\begin{lemma}\label{l4.1}
Let $\Phi$ satisfy Assumption $(\Phi)$. Then $\tx$ is a dense
subspace of $\ttx$ and there exists a positive constant $C$ such
that for all $f\in\tx$, $\|f\|_\ttx\le C\|f\|_\tx$.
\end{lemma}

\begin{proof}
Let $f\in\tx$. By Theorem \ref{l3.1}, there exist
$\tx$-atoms $\{a_j\}_{j=1}^\fz$ and
$\{\lz_j\}_{j=1}^\fz\subset\cc$ such that \eqref{3.1} and
\eqref{3.2} hold.

For any $L\in\nn$, set $\sz_L\ev\sum_{j=1}^L|\lz_j|$. Since $\Phi$
is of upper type $1$, by this together with $\rho(t)=t^{-1}
/\Phi^{-1}(t^{-1})$ for all $t\in(0,\fz)$, we obtain
$$\sum_{j=1}^\fz \mu(B_j)\Phi\lf(\frac{|\lz_j|}{\sz_L \mu(B_j)\rho(\mu(B_j))}\r)
\ge\sum_{j=1}^L\mu(B_j)\Phi\lf(\frac{1}{\sz_L \mu(B_j)\rho(\mu(B_j))}\r)
\frac{|\lz_j|}{\sz_L}\gs 1,$$
which implies that
$$\sum_{j=1}^L|\lz_j|\ls\blz(\{\lz_ja_j\}_{j=1}^\fz)\ls\|f\|_\tx.$$
Letting $L\to\fz$, we further conclude that $\sum_{j=1}^\fz|\lz_j|\ls\|f\|_\tx$.

Since $f\in\tx$ and $\txx=\txz$, we see that
$$f\in\tx\st(\txx)^*=(\txz)^*.$$
Thus, $f\in(\txz)^*$ and $\|f\|_{(\txz)^*}\ls\|f\|_\tx$. Recall that for
any $\ell\in(\txz)^*$, its $(\txz)^*$ norm is defined by
$$\|\ell\|_{(\txz)^*}=\sup_{\|g\|_\txz\le1}|\ell(g)|.$$
Observe also that $a_j\in(\txz)^*$ for all $j\in\nn$. Now, from these
observations, the monotone convergence theorem and H\"older's inequality,
it follows that
\begin{eqnarray*}
\lf\|f-\sum_{j=1}^L\lz_ja_j\r\|_{(\txz)^*}
&&\!=\!\sup_{\|g\|_\txz\le1}\lf|\int_\xt\lf[f(x,t)-\sum_{j=1}^L\lz_ja_j
(x,t)\r]g(x,t)\dxt\r|\\
&&\le\sup_{\|g\|_\txz\le1}\int_\xt\sum_{j= L+1}^\fz|\lz_j||a_j(x,t)g(x,t)|\dxt\\
&&=\sup_{\|g\|_\txz\le1}\sum_{j= L+1}^\fz|\lz_j|\int_{\wh{B_j}}|a_j(x,t)g(x,t)|\dxt\\
&&\le
\sup_{\|g\|_\txz\le1}\sum_{j= L+1}^\fz|\lz_j|\|a_j\|_\txe\|g\chi_{\wh{B_j}}\|_\txe
\le\sum_{j= L+1}^\fz|\lz_j|\to0,
\end{eqnarray*}
as $L\to\fz$. Thus, the series in \eqref{3.1} converges in $(\txz)^*$, which
further implies that $f\in\ttx$ and
$\|f\|_\ttx\le\sum_{j=1}^\fz|\lz_j|\ls\|f\|_\tx.$
This finishes the proof of Lemma \ref{l4.1}.
\end{proof}

\begin{lemma}\label{l4.2}
Let $\Phi$ satisfy Assumption $(\Phi)$. Then $\txb$ is dense in $\ttx$.
\end{lemma}

\begin{proof}
Since $\tx$ is dense in $\ttx$, to prove this lemma, it suffices to prove that
$\txb$ is dense in $\tx$ in the norm $\|\cdot\|_\ttx$.

Fix $x_0\in\cx$. For any $g\in\tx$ and $k\in\nn$, let $g_k\ev g\chi_{O_k}$, where
$$O_k\ev\{(x,t)\in\xt:\ \dist(x,x_0)<k,\ t\in(1/k,k)\}.$$
By the dominated convergence theorem and the continuity of $\Phi$,
we conclude that for any $\lz>0$,
$$\lim_{k\to\fz}\int_\cx\Phi\lf(\frac{\ca(g-g_k)(x)}\lz\r)\,d\mu(x)
=\int_\cx\lim_{k\to\fz}\Phi\lf(\frac{\ca(g-g_k)(x)}\lz\r)\,d\mu(x)=0,$$
which implies that $\lim_{k\to\fz}\|g-g_k\|_\ttx=0$. Then, by Lemma
\ref{l4.1}, we see that
$$\|g-g_k\|_\ttx\ls\|g-g_k\|_\tx\to0,$$
as $k\to\fz$, which completes the proof of Lemma \ref{l4.2}.
\end{proof}

\begin{lemma}\label{l4.3}
Let $\Phi$ satisfy Assumption $(\Phi)$. Then $(\ttx)^*=\txz$
via the pairing
$$\la f,g\ra\to\int_\xt f(y,t)g(y,t)\dyt$$
for all $f\in\ttx$ and $g\in\txz$.
\end{lemma}

\begin{proof}
By Proposition \ref{p4.1} and the definition of $\ttx$, we see that
$(\tx)^*=\txz$ and $\tx\st\ttx$, which further implies that $(\ttx)^*\st\txz$.

Conversely, let $g\in\txz$. Then for any $f\in\ttx$, choose a sequence
$\tx$-atoms $\{a_j\}_{j=1}^\fz$ and $\{\lz_j\}_{j=1}^\fz\subset\cc$
such that $f=\sum_j\lz_ja_j$ in $(\txz)^*$ and $\sum_j|\lz_j|\ls \|f\|_\ttx$.
Thus, by H\"older's inequality, we obtain
\begin{eqnarray*}
|\la f,g\ra|&&\le\sum_j\int_\xt|a_j(x,t)g(x,t)|\dxt\\
&&\le \|g\|_\txz\sum_j|\lz_j|\ls\|g\|_\txz\|f\|_\ttx,
\end{eqnarray*}
which implies that $g\in(\ttx)^*$, and hence completes the proof of
Lemma \ref{l4.3}.
\end{proof}

\begin{lemma}\label{l4.4}
Let $\Phi$ satisfy Assumption $(\Phi)$. If $f\in\ttx$, then
\begin{equation}\label{4.1}
\|f\|_\ttx=\sup_{g\in\txb,\,\|g\|_\txz\le1}\lf|\int_\xt f(x,t)g(x,t)\dxt\r|.
\end{equation}
\end{lemma}

\begin{proof}
Let $f\in\ttx$. From Lemma \ref{4.2}, we deduce that
$$\|f\|_\ttx=\sup_{\|g\|_\txz\le1}\lf|\int_\xt f(x,t)g(x,t)\dxt\r|.$$
Thus, for any $\bz>0$, there exists $g\in\txz$
such that $\|g\|_\txb\le1$ and
$$\lf|\int_\xt f(x,t)g(x,t)\dxt\r|\ge\|f\|_\ttx-\frac\bz2.$$
Observe here that $fg\in L^1(\xt)$. Fix $x_0\in\cx$. Let
$$O_k\ev\{(x,t)\in\xt:\ \dist(x,x_0)<k,\ 1/k<t<k\}.$$
Then there exists $k\in\nn$ such that
$$\lf|\int_\xt f(x,t)g(x,t)\chi_{O_k}\dxt\r|\ge\|f\|_\ttx-\bz.$$
Obviously, $g\chi_{O_k}\in\txb$. Thus, \eqref{4.1} holds, which
completes the proof of Lemma \ref{l4.4}.
\end{proof}

The following lemma is a slight modification of \cite[Lemma 4.2]{cw77};
see also \cite{jya}. We omit the details here.

\begin{lemma}\label{l4.5}
Let $\Phi$ satisfy Assumption $(\Phi)$. Suppose that $\{f_k\}_{k=1}^\fz$
is a bounded family of functions in $\ttx$. Then there exist $f\in\ttx$ and a
subsequence $\{f_{k_j}\}_{j=1}^\fz$
of $\{f_{k}\}_{k=1}^\fz$ such that for all $g\in\txb$,
$$\lim_{j\to\fz}\int_\xt f_{k_j}(x,t)g(x,t)\dxt=\int_\xt f(x,t)g(x,t)\dxt.$$
\end{lemma}

\begin{theorem}\label{t4.1}
Let $\Phi$ satisfy Assumption $(\Phi)$. Then $(\txv)^*$, the
dual space of the space $\txv$, coincides with $\ttx$ in the following
sense:

For any $g\in\ttx$, define the linear function $\ell$ by setting,
for all $f\in\txz$,
\begin{equation}\label{4.2}
\ell(f)\ev\int_\xt f(x,t)g(x,t)\dxt.
\end{equation}
Then there exists a positive constant $C$, independent of $g$, such that
$$\|\ell\|_{(\txz)^*}\le C\|g\|_\ttx.$$

Conversely, for any $\ell\in(\txz)^*$, there exists $g\in\ttx$ such that
\eqref{4.2} holds for all $f\in\txz$ and $\|g\|_\ttx\le C\|\ell\|_{(\txz)^*}$,
where $C$ is a positive constant independent of $\ell$.
\end{theorem}

\begin{proof}
From Lemma \ref{4.2}, we infer that $\txv\st\txz=(\ttx)^*$, which further
implies that $\ttx\st(\ttx)^*\st(\txv)^*$.

Conversely, let $\ell\in(\txv)^*$. Notice that for any $f\in\txb$,
without loss of generality,
we may assume that $\supp f\st K$, where $K$ is a bounded set in $\xt$.
Then we have
$\|f\|_\txv=\|f\|_\txz\le C(K)\|f\|_\txb.$
Thus, $\ell$ induces a bounded linear functional on $\txb$. Let $O_k$ be
as in the proof of Lemma \ref{l4.4}.
By the Riesz representation theorem, there exists
a unique $g_k\in L^2(O_k)$ such that for all $f\in L^2(O_k)$,
$$\ell(f)=\int_\xt f(x,t)g_k(x,t)\dxt.$$
Obviously, $g_{k+1}O_k=g_k$ for all $k\in\nn$. Let $g\ev g_1\chi_{O_1}
+\sum_{k=2}^\fz g_k\chi_{O_k\bh O_{k-1}}$. Then $g\in L_{\rm loc}^2(\xt)$
and for any $f\in\txb$, we have
$$\ell(f)=\int_\xt f(y,t)g(y,t)\dyt.$$

Set $\wz g_k\ev g\chi_{O_k}$. Then for each $k\in\nn$, obviously, we
see that $\wz g_k\in\txb\st\tx\st\ttx$. Then from Lemma \ref{l4.4},
it follows that
\begin{eqnarray*}
\|\wz g_k\|_\ttx&&=\sup_{f\in\txb,\,\|f\|_\txz\le1}
\lf|\int_\xt f(x,t)g(x,t)\chi_{O_k}(x,t)\dxt\r|\\
&&=\sup_{f\in\txb,\,\|f\|_\txz\le1}\lf|\ell(f\chi_{O_k})\r|\\
&&\le\sup_{f\in\txb,\,\|f\|_\txz\le1}\|\ell\|_{(\txv)^*}\|f\|_\txz
\le \|\ell\|_{(\txv)^*}.
\end{eqnarray*}
Thus, by Lemma \ref{l4.5}, there exist $\wz g\in\ttx$ and
$\{\wz g_{k_j}\}_{j=1}^\fz\st\{\wz g_k\}_{k=1}^\fz$ such that for all $f\in\txb$,
$$\lim_{j\to\fz}\int_\xt f(x,t)\wz g_{k_j}(x,t)\dxt=
\int_\xt f(x,t)\wz{g}(x,t)\dxt.$$
On the other hand, notice that for sufficient large $k_j$, we have
\begin{eqnarray*}
\ell(f)&&=\int_\xt f(x,t)g(x,t)\dxt\\
&&=\int_\xt f(x,t)\wz g_{k_j}(x,t)\dxt=\int_\xt f(x,t)\wz{g}(x,t)\dxt,
\end{eqnarray*}
which implies that $g=\wz g$ almost everywhere, and hence $g\in\ttx$.
By a density argument, we conclude that \eqref{4.2} also holds for
$g$ and all $f\in\txz$, which completes the proof of Theorem \ref{t4.1}.
\end{proof}

\begin{definition}\rm\label{d4.2}
Let $L$ satisfy Assumptions $(L)_1$ and $(L)_2$, $\Phi$ satisfy
Assumption $(\Phi)$, $M\in\nn$, $M>\frac n2(\frac1{p_\Phi^-}-\frac12)$
and $\ez\in(n(1/p_\Phi^-- 1/p_\Phi^+),\fz)$. An element
$f\in(\bmox)^*$ is said to be in the \emph{space} $\hmx$ if there exist
$\{\lz_j\}_{j=1}^\fz\st\cc$ and $(\Phi,\,M,\,\ez)_L$-molecules
$\{\az_j\}_{j=1}^\fz$ such that $f=\sum_{j=1}^\fz\lz_j\az_j$ in
$(\bmox)^*$ and
$$\blz(\{\lz_j\az_j\}_{j=1}^\fz)\ev\inf\lf\{\lz>0:\, \sum_{j=1}^\fz\mu(B_j)
\Phi\lf(\frac{|\lz_j|}{\lz \mu(B_j)\rho(\mu(B_j))}\r)\le1\r\}<\fz,$$
where for each $j$, $\az_j$ is adapted to the ball $B_j$.

If $f\in\hmx$, then its \emph{norm} is defined by $\|f\|_\hmx\ev\inf\{\blz
(\{\lz_j\az_j\}_{j=1}^\fz)\}$, where the infimum is taken over all the
possible decompositions of $f$ as above.
\end{definition}

By \cite[Theorem 5.1]{jy}, we see that for all $M>\mz$ and
$\ez\in(n(1/p_\Phi^-- 1/p_\Phi^+),\fz)$, the \emph{spaces $\hx$ and $\hmx$
coincide with equivalent norms}.

Let us introduce the Banach completion of the space $\hx$.

\begin{definition}\rm\label{d4.3}
Let $L$ satisfy Assumptions $(L)_1$ and $(L)_2$, $\Phi$ satisfy
Assumption $(\Phi)$, $\ez\in(n(1/p_\Phi^--1/p_\Phi^+),\fz)$ and
$M>\mz$. The \emph{space} $\bmx$ is defined to be the space of all
$f=\sum_{j=1}^\fz\lz_j\az_j$ in $(\bmox)^*$, where
$\{\lz_j\}_{j=1}^\fz\subset\cc$ with $\sum_{j=1}^\fz|\lz_j|<\fz$
and $\{\az_j\}_{j=1}^\fz$ are
$(\Phi,\,M,\,\ez)_L$-molecules. If $f\in\bmx$, define
$\|f\|_\bmx\ev\inf\{\sum_{j=1}^\fz|\lz_j|\}$, where the infimum is
taken over all the possible decomposition of $f$ as above.
\end{definition}

By \cite[Lemma 3.1]{hm09}, we know that $\bmx$ is a Banach space.
Moreover, from Definition \ref{d4.2}, it is easy to deduce that
$\hx$ is dense in $\bmx$. More precisely, we have the following
lemma.

\begin{lemma}\label{l4.6}
Let $L$ satisfy Assumptions $(L)_1$ and $(L)_2$, $\Phi$ satisfy
Assumption $(\Phi)$, $\ez\in(n(1/p_\Phi^--1/p_\Phi^+),\fz)$ and
$M>\mz$. Then

{\rm i)} $\hx\st\bmx$ and the inclusion is continuous.

{\rm ii)} For any $\ez_1\in(n(1/p_\Phi^--1/p_\Phi^+),\fz)$ and $M_1>\frac
n2(\frac1{p_\Phi^-}-\frac12)$, the spaces $\bmx$ and
$B_{\Phi,L}^{M_1,\ez_1}(\cx)$ coincide with equivalent norms.
\end{lemma}

\begin{proof}
From Definition \ref{d4.3} and the molecular characterization of $\hx$,
it is easy to deduce i).

Let us prove ii). By symmetry, it suffices to
show that $\bmx\st\byx$.
Let $f\in\bmx$. By Definition \ref{d4.3}, there exist $\pme$-molecules
$\{\az_j\}_{j=1}^\fz$ and $\{\lz_j\}_{j=1}^\fz\st\cc$ such that
$f=\sum_{j=1}^\fz\lz_j\az_j$ in $\bmoxx$ and $\sum_{j=1}^\fz|\lz_j|\ls
\|f\|_\bmx$.
By i), for each $j\in\nn$, we see that $\az_j\in\hx\st\byx$ and
$\|\az_j\|_\byx\ls\|\az_j\|_\hx\ls1$. Since $\byx$ is a Banach space, we
see that $f\in\byx$ and $\|f\|_\byx\le\sum_{j=1}^\fz|\lz_j|\|\az_j\|
_\byx\ls\|f\|_\bmx$. Thus, $\bmx\st\byx$, which completes the proof of
Lemma \ref{l4.6}.
\end{proof}

Since the spaces $\bmx$ coincide for all
$\ez\in(n(1/p_\Phi^--1/p_\Phi^+), \fz)$ and $M>\frac
n2(\frac1{p_\Phi^-}-\frac12)$, in what follows, we denote $\bmx$ simple
by $\bx$.

\begin{lemma}\label{l4.7}
Let $L$ satisfy Assumptions $(L)_1$ and $(L)_2$, and $\Phi$ satisfy
Assumption $(\Phi)$. Then $(\bx)^*=\bmox$.
\end{lemma}

\begin{proof}
Since $(\hx)^*=\bmox$ and $\hx\st\bx$, by duality, we conclude that
$(\bx)^*\st\bmox$.

Conversely, let $\ez\in(n(1/p_\Phi^--1/p_\Phi^+),\fz)$, $M>\mz$ and
$f\in\bmox$. For any $g\in\bx$, by Definition \ref{d4.3}, there exist
$\pme$-molecules $\{\az_j\}_{j=1}^\fz$ and $\{\lz_j\}_{j=1}^\fz\st\cc$ such that
$g=\sum_{j=1}^\fz\lz_j\az_j$ in $\bmoxx$ and $\sum_{j=1}^\fz|\lz_j|\ls
\|g\|_\bx$. Thus,
\begin{eqnarray*}
|\la f,g\ra|&&\le\sum_{j=1}^\fz|\lz_j||\la f,\az_j\ra|
\ls\sum_{j=1}^\fz|\lz_j|\|f\|_\bmox\|\az_j\|_\hx\\
&&\ls\|f\|_\bmox\|g\|_\bx,
\end{eqnarray*}
which implies that $f\in(\bx)^*$, and hence completes the proof of
Lemma \ref{l4.7}.
\end{proof}

Let $M\in\nn$. For all $F\in L^2(\xt)$ with bounded support, define
\begin{equation}\label{4.3}
\plm F\ev C(M)\int_0^\fz\tml F(\cdot,t)\dt,
\end{equation}
where $C(M)$ is as in \eqref{3.5}.

\begin{proposition}\label{p4.2}
Let $L$ satisfy Assumptions $(L)_1$ and $(L)_2$, $\Phi$ satisfy
Assumption $(\Phi)$ and $M\in\nn$. Then the operator $\plm$,
initially defined on $\txb$, extends to a bounded linear operator

{\rm i)} from $T_2^2(\cx)$ to $L^2(\cx)$;

{\rm ii)} from $\tx$ to $\hx$, if $M>\frac n2(\frac1{p_\Phi^-}-\frac12)$;

{\rm iii)} from $\ttx$ to $\bx$, if $M>\frac n2(\frac1{p_\Phi^-}-\frac12)$;

{\rm iv)} from $\txv$ to $\vmo$.
\end{proposition}

\begin{proof}
i) and ii) were established in \cite[Proposition 3.6]{al11} (see also
\cite[Lemma 3.1]{jy}).

By Lemma \ref{l4.2}, we know that $\txb$ is dense in $\ttx$. Let $f\in\txb$.
From ii) and Lemma \ref{l4.6}, we deduce that $\plm f\in\hx\st\bx$.
Moreover, by Definition \ref{d4.1}, there exist $\tx$-atoms
$\{a_j\}_{j=1}^\fz$ and $\{\lz_j\}_{j=1}^\fz\st\cc$ such that
$f=\sum_{j=1}^\fz\lz_ja_j$ in $(\txz)^*$ and $\sum_j|\lz_j|\ls
\|f\|_\ttx$. In addition, for any $g\in\bmox$, we have
$\tmlx g\in\txz$. Thus, by $(\tx)^*=\txz$, we conclude that
\begin{eqnarray*}
\la\plm(f),g\ra&&=C(M)\int_\xt f(x,t) \ov{\tmlx g(x)}\dxt\\
&&=\sum_{j=1}^\fz\lz_jC(M)\int_\xt a_j(x,t) \ov{\tmlx g(x)}\dxt\\
&&=\sum_{j=1}^\fz\lz_j\la\plm(a_j),g\ra,
\end{eqnarray*}
which implies that $\plm(f)=\sum_{j=1}^\fz\lz_j\plm(a_j)$ in $(\bmox)^*$.
By ii), we further conclude that
\begin{eqnarray*}
\|\plm(f)\|_\bx&&\le\sum_{j=1}^\fz|\lz_j|\|\plm(a_j)\|_\bx\\
&&\ls\sum_{j=1}^\fz|\lz_j|\|\plm(a_j)\|_\hx\ls\|f\|_\ttx.
\end{eqnarray*}
Since $\txb$ is dense in $\ttx$, we see that $\plm$ extends to a
bounded linear operator from $\ttx$ to $\bx$, which completes the
proof of iii).

Let us now prove iv). From Lemma \ref{l3.3}, we infer that $\txb$ is
dense in $\txv$. Thus, to prove iv), it suffices to show that $\plm$
maps $\txb$ continuously into $\vmo$.

Let $f\in\txb$. By i), we see that $\plm f\in L^2(\cx)$. Notice that
\eqref{3.3} and \eqref{3.4} with $L$ and $L^*$ exchanged implies
that $L^2(\cx)\st\mly$, when $M_1\in\nn$ and $M_1>\mz$. Thus, $\plm f
\in\mly$. To show $\plm f\in\vmo$, by Theorem \ref{t3.4}, we still need
show that $\tmy\plm f\in\txv$.

For any ball $B\ev B(x_B,r_B)$, let $V_0(B)\ev\wh B$ and $V_k(B)\ev
(\wh{2^kB})\bh(\wh{2^{k-1}B})$ for any $k\in\nn$. For all $k\in\zz_+$,
let $f_k\ev f\chi_{V_k(B)}$. Thus, for $k\in\{0,1,2\}$, by Lemma \ref{l2.2}
and i), we see that
\begin{eqnarray*}
\lf[\iint_{\wh B}\lf|\tmy\plm f_k(x)\r|^2\dxt\r]^{1/2}\ls\|\plm f_k\|_{L^2(\cx)}
\ls\|f_k\|_{T_2^2(\cx)}.
\end{eqnarray*}
For $k\ge3$, let $\vy\ev(\wh{2^kB})\bh(2^{k-2}B\times(0,\fz))$ and
$\ve\ev V_k(B)\bh\vy$. We further write
$f_k=f_k\chi_\vy+f_k\chi_\ve\ev f_{k,1}+f_{k,2}$. From Minkowski's
inequality, Lemma \ref{l2.3} and H\"older's inequality, we deduce that
\begin{eqnarray*}
&&\lf[\iint_{\wh B}\lf|\tmy\plm f_{k,2}(x)\r|^2\dxt\r]^{1/2}\\
&&\hs\sim\lf[\iint_{\wh B}\lf|\int_{2^{k-2}r_B}^{2^kr_B}\tmy
(s^2L)^Me^{-s^2L}(f_{k,2}(\cdot,s))(x)\ds\r|^2\dxt\r]^{1/2}\\
&&\hs\ls\int_{2^{k-2}r_B}^{2^kr_B}\lf[\iint_{\wh B}\lf|t^{2M_1}
s^{2M}L^{M+M_1}e^{-(s^2+t^2)L}(f_{k,2}(\cdot,s))(x)\r|^2\dxt\r]^{1/2}\ds\\
&&\hs\ls\int_{2^{k-2}r_B}^{2^kr_B}\lf[\int_0^{r_B}\lf|\frac{t^{2M_1}
s^{2M}}{(s^2+t^2)^{M+M_1}}\r|^2\|f_{k,2}(\cdot,s)\|_{L^2(\cx)}^2\dt\r]^{1/2}\ds\\
&&\hs\ls2^{-2kM_1}\int_{2^{k-2}r_B}^{2^kr_B}\|f_{k,2}(\cdot,s)\|_{L^2(\cx)}\ds
\ls2^{-2kM_1}\|f_{k,2}\|_{T_2^2(\cx)}.
\end{eqnarray*}
Similarly, we have
$$\lf[\iint_{\wh B}\lf|\tmy\plm f_{k,1}(x)\r|^2\dxt\r]^{1/2}
\ls2^{-2kM_1}\|f_{k,1}\|_{T_2^2(\cx)}.$$

Let $\wz p_\Phi\in (0,p_\Phi^-)$
such that $M>\mzx$ and $M_1>\mzx$.
Combining the above estimates, since $\Phi$ is of lower type $\wz p_\Phi$,
we finally conclude that
\begin{eqnarray*}
&&\frac1\rb\lf[\iint_{\wh B}\lf|\tmy\plm f(x)\r|^2\dxt\r]^{1/2}\\
&&\hs\ls
\sum_{k=0}^2\frac1\rb\lf[\iint_{\wh B}\lf|\tmy\plm f_k(x)\r|^2\dxt\r]^{1/2}\\
&&\hs\hs+\sum_{k=3}^\fz\sum_{i=1}^2\frac1\rb\lf[\iint_{\wh B}
\lf|\tmy\plm f_{k,i}(x)\r|^2\dxt\r]^{1/2}\\
&&\hs\ls\sum_{k=0}^2\frac1\rb\|f_k\|_{T_2^2(\cx)}
+\sum_{k=3}^\fz\sum_{i=1}^2\frac{2^{-2kM_1}}\rb\|f_{k,i}\|_{T_2^2(\cx)}\\
&&\hs\ls\sum_{k=0}^\fz2^{-2k[M_1-\mzx]}\frac1\rkb\|f_k\|_{T_2^2(\cx)}.
\end{eqnarray*}
Since $f\in\txv\st\txz$, we have
$$\frac1\rkb\|f_k\|_{T_2^2(\cx)}\ls\|f\|_\txz$$
and, for all fixed $k\in\nn$,
\begin{eqnarray*}
\lim_{c\to0}\sup_{B:\,r_B\le c}\frac{\|f_k\|_{T_2^2(\cx)}}\rkb&&
=\lim_{c\to\fz}\sup_{B:\,r_B\ge c}\frac{\|f_k\|_{T_2^2(\cx)}}\rkb\\
&&=\lim_{c\to\fz}\sup_{B:\,B\st [B(0,c)]^\com}\frac{\|f_k\|_{T_2^2(\cx)}}\rkb=0.
\end{eqnarray*}
Thus, by the dominated convergence theorem for series, we further conclude that
\begin{eqnarray*}
&&\eta_1(\tmy\plm f)\\
&&\hs=\lim_{c\to0}\sup_{B:\,r_B\le c}\frac1\rb
\lf[\iint_{\wh B}\lf|\tmy\plm f(x)\r|^2\dxt\r]^{1/2}\\
&&\hs\ls\sum_{k=0}^\fz2^{-2k[M_1-\mzx]}
\lim_{c\to0}\sup_{B:\,r_B\le c}\frac{\|f_k\|_{T_2^2(\cx)}}\rkb=0.
\end{eqnarray*}

Similarly, we have $\eta_2(\tmy\plm f)=\eta_3(\tmy\plm f)=0$, and
hence $\tmy\plm f\in\txv$, which completes the proof of Proposition \ref{p4.2}.
\end{proof}

\begin{lemma}\label{l4.8}
Let $L$ satisfy Assumptions $(L)_1$ and $(L)_2$, and $\Phi$ satisfy
Assumption $(\Phi)$. Then $\vmo\cap L^2(\cx)$ is dense in $\vmo$.
\end{lemma}

\begin{proof}
Let $f\in\vmo$ and $M>\mz$. Then by Theorem \ref{t3.4}, we have
$h\ev\tml f \in\txv$. Similarly to the proof of Proposition \ref{p4.2},
by Lemma \ref{l3.3}, there exist $\{h_k\}_{k\in\nn}\st\txb\st\txv$
such that $\|h-h_k\|_\txz\to0$, as $k\to\fz$. Thus, by i) and iv)
of Proposition \ref{p4.2}, we see that $\ply h_k\in L^2(\cx)\cap\vmo$ and
\begin{equation}\label{4.4}
\|\ply(h-h_k)\|_\bmo\ls\|h-h_k\|_\txz\to0,
\end{equation}
as $k\to\fz$.

Let $\az$ be a $\pme$-molecule. Then by the definition of $\hx$, we
know that $e^{-t^2L}\az\in\tx$, which, together with Lemma \ref{l3.2},
the fact that $(\tx)^*=\txz$ and $(\hx)^*=\bmo$, further implies that
\begin{eqnarray*}
\la f,\az\ra&&=C(M)\iint_\xt\tml f(x) \ov{t^2L^*e^{-t^2L^*}\az(x)}\dxt\\
&&=\lim_{k\to\fz} C(M)\iint_\xt h_k(x) \ov{t^2L^*e^{-t^2L^*}\az(x)}\dxt\\
&&=\frac{C(M)}{C_1}\lim_{k\to\fz}\int_\cx(\ply h_k(x)) \ov{\az(x)}\,d\mu(x)
=\frac{C(M)}{C_1}\la\ply h,\az\ra.
\end{eqnarray*}
Since the set of finite combinations of molecules is dense in $\hx$,
we then see that $f=\frac{C(M)}{C_1}\ply h$ in $\bmo$.

Now, for each $k\in\nn$, let $f_k\ev\frac{C(M)}{C_1}\ply h_k$. Then $f_k
\in\vmo\cap{L^2(\cx)}$ and, moreover, by \eqref{4.4},
we have $\|f-f_k\|_\bmo\to0$,
as $k\to\fz$, which completes the proof of Lemma \ref{l4.8}.
\end{proof}

In what follows, the \emph{symbol} $\la\cdot,\cdot\ra$ in the following
theorem means the duality between the space $\bmo$ and the space
$\bxx$ in the sense of Lemma \ref{l4.7} with $L$ and $L^*$ exchanged.

\begin{theorem}\label{t4.2}
Let $L$ satisfy Assumptions $(L)_1$ and $(L)_2$, and $\Phi$ satisfy
Assumption $(\Phi)$. Then the dual space of $\vmo$, $(\vmo)^*$,
coincides with the space $\bxx$ in the following sense:

For any $g\in\bxx$, define the linear functional $\ell$ by setting,
for all $f\in\vmo$,
\begin{equation}\label{4.5}
\ell(f)\ev\la f,g\ra.
\end{equation}
Then there exists a positive constant $C$ independent of $g$ such that
$$\|\ell\|_\vmox\le C\|g\|_\bxx.$$

Conversely, for any $\ell\in\vmox$, there exist $g\in\bxx$ such
that \eqref{4.5} holds and a positive constant $C$,
independent of $\ell$, such that
$$\|g\|_\bxx\le C\|\ell\|_\vmox.$$
\end{theorem}

\begin{proof}
By Lemma \ref{l4.7}, we have $(\bxx)^*=\bmo$. Definition \ref{d3.3}
implies that $\vmo\st\bmo$, which further implies that $\bxx\st(\vmo)^*$.

Conversely, let $M>\mz$ and $\ell\in(\vmo)^*$. By Proposition
\ref{p4.2}, $\ply$ is bounded from $\txv$ to $\vmo$, which implies
that $\ell\circ\ply$ is a bounded linear functional on $\txv$. Thus,
by Theorem \ref{t4.1}, there exists $g\in\ttx$ such that for all
$g\in\txv$, $\ell\circ\ply(f)=\la f,g\ra$.

Now, suppose that $f\in\vmo\cap L^2(\cx)$. By Theorem \ref{t3.4}, we conclude
that $\tml f\in\txv$. Moreover, from the proof of Lemma \ref{l4.8}, we deduce
that $f=\frac{C(M)}{C_1}\ply(\tml f)$ in $\bmo$. Thus
\begin{eqnarray}\label{4.6}
\ell(f)&&=\frac{C(M)}{C_1}\ell\circ\ply(\tml f)\\
&&=\frac{C(M)}{C_1}\iint_\xt\tml f(x)g(x,t)\dxt.\noz
\end{eqnarray}
By Lemma \ref{l4.2}, $\txb$ is dense in $\ttx$. Since $g\in\ttx$, we choose
$\{g_k\}_{k\in\nn}\st\txb$ such that $g_k\to g$ in $\ttx$. By iii) of Proposition
\ref{p4.2}, we see that $\pi_{L^*,M}(g)$, $\pi_{L^*,M}(g_k)\in
B_{\Phi,L^*}(\cx)$ and
$$\|\pi_{L^*,M}(g-g_k)\|_{B_{\Phi,L^*}(\cx)}\ls\|g-g_k\|_\ttx\to0,$$
as $k\to\fz$. This, together with \eqref{4.6}, Theorem \ref{t4.1},
the dominated convergence theorem and Lemma \ref{l4.7}, implies that
\begin{eqnarray}\label{4.7}
\ell(f)&&=\frac{C(M)}{C_1}\lim_{k\to\fz}\iint_\xt\tml f(x)g_k(x,t)\dxt\\
&&=\frac{C(M)}{C_1}\lim_{k\to\fz}\int_\cx f(x)\int_0^\fz\tmlx
(g_k(\cdot,t))(x)\dt\,d\mu(x)\noz\\
&&=
\frac1{C_1}\lim_{k\to\fz}\la f,\pi_{L^*,M}(g_k)\ra
=\frac1{C_1}\la f,\pi_{L^*,M}(g)\ra.\noz
\end{eqnarray}
Since $\vmo\cap{L^2(\cx)}$ is dense in $\vmo$, we finally conclude
that \eqref{4.7} holds for all $f\in\vmo$, and
$\|\ell\|_\vmox=\frac1{C_1}\|\pi_{L^*,M}g\|_\bxx$. In this sense, we
have $\vmox\st\bxx$, which completes the proof of Theorem
\ref{t4.2}.
\end{proof}

\noindent{\bf Acknowledgements.} The authors would like to
thank Dr. Bui The Anh and Dr. Renjin Jiang
for some helpful discussions on the subject of this paper
and, especially, for Dr. Bui The Anh to give us his preprint \cite{a10}.
The authors would also like to thank the referee for his/her several
valuable remarks which made this article more readable.

\bigskip

Yiyu Liang, Dachun Yang  and Wen Yuan (Corresponding author)

\medskip

School of Mathematical Sciences, Beijing Normal University,
Laboratory of Mathematics and Complex Systems, Ministry of
Education, Beijing 100875, People's Republic of China

\smallskip

{\it E-mails}: \texttt{yyliang@mail.bnu.edu.cn} (Y. Liang)

\hspace{1.55cm}\texttt{dcyang@bnu.edu.cn} (D. Yang)

\hspace{1.55cm}\texttt{wenyuan@bnu.edu.cn} (W. Yuan)
\end{document}